\documentclass[10pt]{article}
\usepackage{amsmath,mathtools}

\usepackage[font=footnotesize]{caption}
\usepackage{subcaption}

\usepackage{graphicx}
\usepackage{amsfonts}
\usepackage{setspace}
\usepackage{xcolor}
\usepackage{color}

\usepackage[T1]{fontenc}
\usepackage{mathpazo}

\usepackage{sectsty}

\usepackage[numbers,sort&compress]{natbib}

\usepackage{amsthm}
\usepackage{epsfig,epsf,psfrag}
\newcommand{\er}{\mathbf{e}_r}
\newcommand{\et}{\mathbf{e}_{\theta}}
\newcommand{\ez}{\mathbf{e}_z}

\newcommand{\exx}{\mathbf{e}_1}
\newcommand{\eyy}{\mathbf{e}_2}
\newcommand{\ezz}{\mathbf{e}_3}

\usepackage{bm}
\newcommand{\vct}[1]{\bm{\mathsf{#1}}}

\newcommand{\mtx}[1]{\bm{\mathsf{#1}}}

\newcommand{\bn}{\mathbf{n}}
\newcommand{\bbb}{\mathbf{b}}
\newcommand{\btau}{\boldsymbol{\tau}}
\newcommand{\bt}{\mathbf{t}}
\newcommand{\bx}{\mathbf{x}}
\newcommand{\by}{\mathbf{y}}

\newcommand{\gccm}{G_m^{\cos}}
\newcommand{\gssm}{G_m^{\sin}}
\newcommand{\gm}{G_m}

\newcommand{\bzero}{\mathbf{0}}

\newcommand{\be}{\mathbf{E}}
\newcommand{\beinc}{\mathbf{E^{\text{inc}}}}
\newcommand{\besc}{\mathbf{E^{\text{sc}}}}
\newcommand{\beint}{\mathbf{E^{\text{int}}}}
\newcommand{\beext}{\mathbf{E^{\text{ext}}}}
\newcommand{\beinte}{\mathbf{E_{\text{e}}^{\text{int}}}}
\newcommand{\beexte}{\mathbf{E_{\text{e}}^{\text{ext}}}}

\newcommand{\cS}{\mathcal S}

\newcommand{\Etinc}{\tilde{\mathbf{E}}^{\text{inc}}}
\newcommand{\Htinc}{\tilde{\mathbf{H}}^{\text{inc}}}
\newcommand{\Etsc}{\tilde{\mathbf{E}}^{\text{sc}}}
\newcommand{\Htsc}{\tilde{\mathbf{H}}^{\text{sc}}}
\newcommand{\Et}{\tilde{\mathbf{E}}}
\newcommand{\Ht}{\tilde{\mathbf{H}}}

\newcommand{\Einc}{\mathbf{E}^{\text{inc}}}
\newcommand{\Hinc}{\mathbf{H}^{\text{inc}}}

\newcommand{\bh}{\mathbf{H}}
\newcommand{\bhsc}{\mathbf{H^{\text{sc}}}}
\newcommand{\bhinc}{\mathbf{H^{\text{inc}}}}
\newcommand{\bhint}{\mathbf{H^{\text{int}}}}
\newcommand{\bhext}{\mathbf{H^{\text{ext}}}}
\newcommand{\bhinte}{\mathbf{H_{\text{e}}^{\text{int}}}}
\newcommand{\bhexte}{\mathbf{H_{\text{e}}^{\text{ext}}}}

\newcommand{\bj}{\mathbf{J}} 
\newcommand{\bbm}{\mathbf{M}}
\newcommand{\bjs}{\mathbf{j}} 
\newcommand{\bbms}{\mathbf{m}}

\newcommand{\mk}{\mathcal{K}}
\newcommand{\mn}{\mathcal{N}}
\newtheorem{theorem}{Theorem}[section]
\newtheorem{lemma}[theorem]{Lemma}

\newtheorem{corollary}[theorem]{Corollary}

\newtheorem{remark}{Remark}
\newtheorem{definition}{Definition}

\title{Robust integral formulations for electromagnetic scattering 
 from three-dimensional cavities} 

\numberwithin{equation}{section}

\author{Jun Lai\footnote{Courant Institute, New York University, New
    York, NY. Email: {\tt lai@cims.nyu.edu}}, \,
Leslie Greengard\footnote{Courant Institute, New York University,
  and Simons Center for Data Analysis, Simons Foundation, New York, NY. Email:
  {\tt greengard@cims.nyu.edu}}, \,
   and 
Michael O'Neil\footnote{Courant Institute and Tandon School of
  Engineering,  New York University, New
    York, NY. Email: {\tt oneil@cims.nyu.edu}}
}

\date{\today}

\usepackage[margin=1.2in]{geometry}

\begin{document}

\maketitle 
\begin{abstract}
  Scattering from large, open cavity structures is of importance in a
  variety of electromagnetic applications. In this paper, we propose a
  new well conditioned integral equation for scattering from general
  open cavities embedded in an infinite, perfectly conducting
  half-space. The integral representation permits the stable
  evaluation of both the electric and magnetic field, even in the
  low-frequency regime, using the continuity equation in a
  post-processing step. We establish existence and uniqueness results,
  and demonstrate the performance of the scheme in the
  cavity-of-revolution case. High-order accuracy is obtained using a
  Nystr\"om discretization with generalized Gaussian quadratures.
\end{abstract}

\onehalfspacing

\section{Introduction}
\label{sec_intro}

The computation of electromagnetic wave propagation in the presence of
large, open cavities is an important modeling task.  It is critical,
for example, in understanding the effect of exhaust nozzles and engine
inlets on aircraft, as well as surface deformations in automobiles and
other land-based vehicles~\cite{GW1,GJP,PLA15,JIN2,JL2014}. The
presence of such structures plays a dominant role in both the near field,
where electromagnetic interference is of concern, and in the far field,
where the radar cross-section can be used for identification and
classification (including stealth-related calculations).  Fast
and accurate solvers to simulate such scattering phenomena are essential
for both design optimization and verification.

A variety of numerical methods have been proposed to solve such
scattering problems. Largely speaking, they fall into two categories.
The first is direct numerical simulation using finite
difference~\cite{GW1}, finite element~\cite{JIN2},
mode-matching~\cite{GLin} and boundary integral
methods~\cite{GW1,Perez-Arancibia2014}.  The second is asymptotic
methods, including Gaussian beam approximations~\cite{GBS} and
physical optics-based schemes~\cite{SHT}. The latter methods tend to
work well at very high frequencies in the absence of multiple
near-field scattering events, and are generally not well suited for
high-precision calculations in geometrically complex environments.
Solving the governing Maxwell equations using finite difference and
finite element methods, on the other hand, requires the discretization
of an unbounded domain.  In practice, these methods must either employ
approximate outgoing boundary conditions to mimic the radiation
condition at infinity, or be coupled with a boundary integral
representation beyond some distance so that the radiation condition is
satisfied exactly.  In the present work, we will focus on boundary
integral equation methods since they are free from grid-based
numerical dispersion, can achieve high-order accuracy in complex
geometry, and require degrees of freedom only on the boundary of the
scatterer itself, greatly reducing the number of unknowns.  The
Green's function used to represent the solution satisfies the outgoing
(radiation) condition exactly.  Existing integral representations for
cavity problems, however, generally yield integral equations of the
first-kind~\cite{HGA2}.  First-kind equations can lead to
ill-conditioned discrete linear systems, especially if substantial
mesh refinement is required. Refined meshes may be needed, for
example, to resolve geometric singularities.  Furthermore, several
existing formulations also suffer from \emph{spurious
  resonances}, including those of mixed first-/second-kind
systems~\cite{Perez-Arancibia2014}.  Finally, because of the nature of
the dyadic Green's function for the electric field,
standard
methods based on discretizing the physical electric current 
also suffer from low-frequency breakdown~\cite{EG10,ZC2000}. This
behavior is discussed in more detail below.

In this paper, we propose a new representation of the scattered field
that leads to a well-posed (resonance-free)  integral
equation which is immune from low-frequency breakdown.  This allows
for a stable numerical discretization along arbitrarily adaptive
meshes. In our numerical examples, using the 
the fact that the scatterer (i.e. the cavity) is axisymmetric 
permits us to use
separation of variables in cylindrical coordinates, applying the
Fourier transform in the angular (azimuthal) variable.
This procedure leads to a
sequence of uncoupled two-dimensional boundary integral equations on
the generating curve that 
defines the cross-section of the boundary of the scatterer
(see Fig.~\ref{figure2}).  
There are various numerical technicalities associated with
implementing body-of-revolution integral equation solvers, and 
  we do not seek to review the substantial
literature here. We instead refer the reader
to~\cite{HK2014,Kucharski2000,Hao2015,YHM2012,Liu2015} and the
references therein. A concise overview of the discretization and
resulting solver is given in Section~\ref{sec_fourier}. 
 Similar high-order techniques have been
applied to solve the Helmholtz equation on surfaces of
revolution~\cite{HK2014,Hao2015,YHM2012,Liu2015} and the full Maxwell
equations (for closed-cavity resonance problems) in~\cite{HK15}.

Due to the applicability of cavity scattering in physics and engineering,
there has been much work dedicated to both the mathematical and numerical
aspects of the problem.
The well-posedness
of the (forward) scattering problem is discussed in~\cite{HGA1,HGA2}
in the case of the two-dimensional problem, and in~\cite{HGA3} for 
the three-dimensional case. 
The paper~\cite{GKZ} provides the explicit dependence of the scattered field 
on the wavenumber in the high-frequency context. In~\cite{GJP,Li2012}, 
the authors studied uniqueness and stability issues for the inverse problem,
where one seeks to recover the shape of an unknown cavity using 
near-field data.
The corresponding optimal design problem, i.e. to find a cavity shape
that minimizes the radar cross section, under certain constraints, 
is discussed in~\cite{BJL1,BJL2} in the two-dimensional setting.

An outline of the paper follows: Section~\ref{sec_form} provides a
detailed introduction to the problem of scattering from an open cavity
and proposes an integral representation that leads to a well conditioned
integral formulation. In Section~\ref{sec_prof}, we prove that this
integral equation has a unique solution for a given incident field.
In Section~\ref{sec_lowfreq}, we show how to avoid low-frequency
breakdown merely by the use of various vector identities and physical
considerations. In Section~\ref{sec_fourier}, we briefly discuss the
separation of variables solver for axisymmetric cavities, and then
illustrate its accuracy and stability in Section~\ref{sec_numeri}.
Section~\ref{sec_con} contains a brief discussion of open problems and
some concluding remarks.

\section{Mathematical formulation of the scattering problem}
\label{sec_form}

\begin{figure}[!t]
\center
\includegraphics[width=5in]{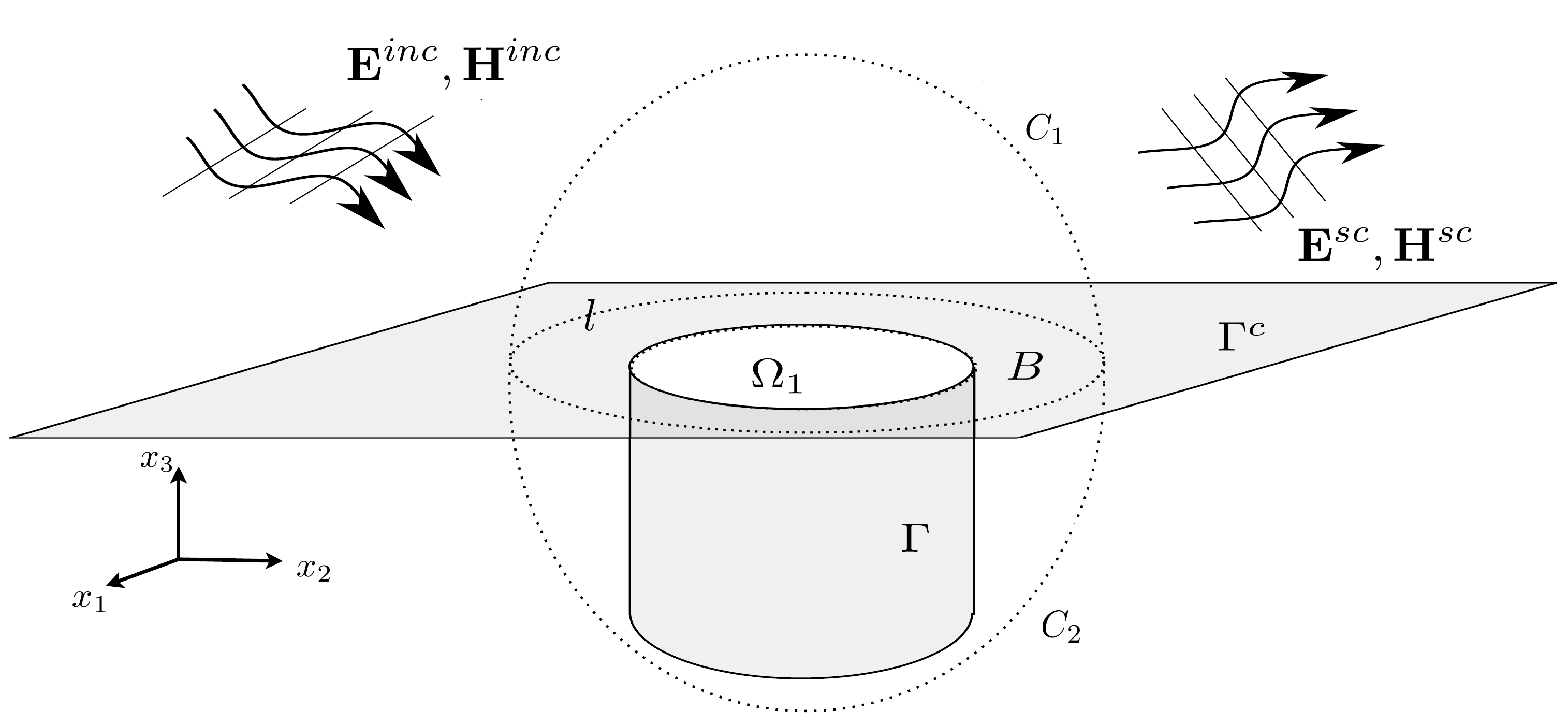}
\caption{A cavity $\Omega_1$ in a perfectly conducting
half-space $x_3 \geq 0$, with boundary $\Gamma$. The surfaces $C_1$ and
$C_2$ define the two halves of a sphere $C$ which is
sufficiently large to contain
the cavity and a finite buffer region, denoted by $B$, in the
$x_1x_2$-plane beyond the edge of the cavity.
The curve $\ell$ denotes the outer edge of the buffer region $B$. 
The unbounded half-space
boundary outside of $C$ is denoted by $\Gamma^c$.}\label{figure1}
\end{figure}

Suppose now that a perfectly conducting cavity $\Omega_1$ extends into the 
lower half-space $x_3 < 0$, as depicted in Figure~\ref{figure1}. See
the caption of Figure~\ref{figure1} for a description of the
geometrical setup.
The region in the lower half-space with boundary $\Gamma \cup B \cup
\Gamma^c$
is assumed to be perfectly conducting.
Given a time harmonic incident 
field $(\Etinc,\Htinc)$
 with an implicit time dependence of 
$e^{-i\omega t}$, we seek to find the scattered
field $(\Etsc,\Htsc)$ so that the total
field 
\[
\Et = \Etinc + \Etsc, \quad
\Ht = \Htinc + \Htsc
\]
satisfies the the Maxwell equations
\begin{equation*}
\begin{aligned}
\nabla \times \Et - i\omega \mu \, \Ht &= 0,  \\
\nabla \times \Ht  + i\omega \varepsilon \, \Et &= 0
\end{aligned}
\end{equation*} 
for $\bx \in \mathbb{R}^{3+}\cup \Omega_1$.
The material parameters are given by $\varepsilon$,  the 
electric permittivity, and $\mu$,  the magnetic permeability. 
Assuming $\varepsilon$ and $\mu$ are constant, 
it is convenient to denote suitably normalized 
fields by  
$\be = \sqrt{\varepsilon} \Et$, $\bh = \sqrt{\mu}\Ht$, etc.,
 leading to a simpler form of Maxwell's equations
\begin{equation}\label{maxequ}
\begin{aligned}
\nabla \times \be - ik \, \bh &= 0,  \\
\nabla \times \bh + ik \, \be &= 0,
\end{aligned}
\end{equation} 
where $k= \omega \sqrt{\mu \varepsilon}$ is known as the wavenumber. 
We assume that $\Re(k) >0$ and $\Im{(k)}\ge 0$. For perfect conductors,  
it is well-known~\cite{papas}
that the tangential electric field must satisfy the boundary conditions 
\begin{equation}\label{perfectbd}
\bn \times \be = \bzero \quad \text{on } \Gamma^c\cup \Gamma,
\end{equation} 
where $\bn$ is the interior normal direction along $\Gamma^c\cup \Gamma$.
The scattered field must also satisfy the Silver-M\"uller 
radiation condition: 
\begin{equation}
\lim_{|\bx|\rightarrow \infty} \frac{1}{|\bx|}
\left( \bhsc \times \frac{\bx}{|\bx|} - \besc  \right) = 0.  
\end{equation}

\begin{remark}
In the following derivations and computations, 
we assume that the incident field $(\beinc,\bhinc)$ is defined so that
it satisfies not only Maxwell's equations, but also the tangential
boundary condition
$\bn\times \beinc = \bzero $ on the entire half-space $x_3 = 0$. 
This is easy to accomplish by reflection and discussed in more detail in
Section~\ref{sec_numeri}.
\end{remark}

Before discussing the solution of the cavity problem itself,
we introduce some necessary notation.
Given a tangential vector field $\bjs$ along some surface $\Gamma$, the 
{\em vector potential} is defined by the single-layer potential
\begin{equation}
\cS_{\Gamma} \bjs (\bx) = \int_{\Gamma} G(\bx,\by) \, \bjs(\by) \, dA_y,
\end{equation}
where $G(\bx,\by)$ is the Green's function for the three-dimensional 
Helmholtz equation 
\begin{equation}
G(\bx,\by) = \frac{e^{ik|\bx-\by|}}{4\pi |\bx-\by|}.
\end{equation}
It is well known~\cite{papas} that in the case where $\bjs=\bj$, a
physical electric current, 
then the corresponding electric and magnetic fields generated by $\bj$ are
given by 
\begin{equation} \label{vectpot1}
\begin{aligned}
\be &= -\frac{1}{ik} \nabla \times \nabla \times
\cS_{\Gamma}\bj,\\
\bh &=  \nabla \times \cS_{\Gamma}\bj.
\end{aligned}
\end{equation}
Likewise, if $\bjs=\bbm$, a surface \emph{magnetic current}, then 
the electric and magnetic fields induced by $\bbm$ are given by the 
{\em vector anti-potentials}
\begin{equation}\label{vectpot2}
\begin{aligned}
\be^{\text a} &= \nabla \times \cS_{\Gamma}\bbm, \\
\bh^{\text a} &= \frac{1}{ik} \nabla \times \nabla \times
\cS_{\Gamma}\bbm.
\end{aligned}
\end{equation}
Due to the linearity of Maxwell's equations, any linear combination of 
electric current-like variable $\bjs$ and magnetic current-like
variable $\bbms$ will generate a Maxwellian field 
\begin{equation}
\begin{aligned}
 \be(\bjs,\bbms) &= \alpha \be(\bjs) + \beta \be^{\text a}(\bbms), \\
\bh(\bjs,\bbms) &= \alpha \bh(\bjs) + \beta \bh^{\text a}(\bbms) .
\end{aligned}
\end{equation}
Only when $\beta = 0$ does $\bjs$ correspond directly to the
\emph{physical} electric current. In order to develop a well-conditioned
integral equation for the cavity problem, 
we will make use of both potentials and anti-potentials in the representation.
Boundary conditions will then determine the values of $\bjs$ 
and $\bbms$. Such an approach is 
sometimes called the \emph{indirect-method}
 since the unknowns are not the fields themselves.

\subsection{A simpler problem: The bump}

It is first worth considering the simpler
problem where the defect in the half-space boundary is a
compactly-supported
\emph{bump} 
instead of a cavity (Fig.~\ref{figure1a}). 
This problem can be solved using standard integral equations and
the method of images.
We need only satisfy boundary 
condition~\eqref{perfectbd}, which we write in the form:
\begin{equation}
\label{bumpbc}
\bn\times \besc = -\bn \times \beinc.
\end{equation}

Since, by assumption, $\bn \times \beinc =\mathbf{0}$ away from the bump, 
we need only satisfy condition~\eqref{bumpbc} 
on the bump itself as long as we can construct
a representation for $\besc$ that satisfies 
$\bn \times \besc =\bzero $ away from the bump. To this end, we can define
\begin{equation}\label{vecpotim}
\besc = \nabla \times \cS_{\Gamma}\bbm + \nabla \times \cS_{\Gamma_R} \bbm_R
\end{equation}
and
\begin{equation}
\bhsc = \frac{1}{ik} \nabla \times \nabla \times \cS_{\Gamma}\bbm +
\frac{1}{ik} \nabla \times \nabla \times \cS_{\Gamma_R}\bbm_R.
\end{equation}
Here, $\Gamma_R$ is the reflection of the bump $\Gamma$ 
across the $x_1x_2$-plane.
If at a point $\bx = (a,b,c) \in \Gamma$, the magnetic current $\bbm$
is given by \mbox{$\bbm{(\bx)} = (m_1,m_2,m_3)$}, then
its image point on $\Gamma_R$ is $\bx' = (a,b,-c)$ and the image 
current is defined as $\bbm_R{(\bx')} = (m_1,m_2,-m_3)$.
It is straightforward to verify from~\eqref{vecpotim} that
$\bn \times \besc =\bzero$ away from the bump (i.e., wherever $\bn =
(0,0,1)$. 
From standard jump condition relations, taking the limit
of~\eqref{vecpotim}  
to the boundary, it remains only to solve
the boundary integral equation for $\bbm$:
\[ \frac{1}{2} \bbm + \bn \times \nabla \times \cS_{\Gamma}\bbm 
+ \bn \times \nabla \times \cS_{\Gamma_R}\bbm_R = - n \times \beinc
\]
on $\Gamma$, where the integral operators are understood in their
principal value sense. This is a second-kind (although not
resonance-free) integral equation for a smooth bump.

\begin{figure}
\center
\includegraphics[width=5in]{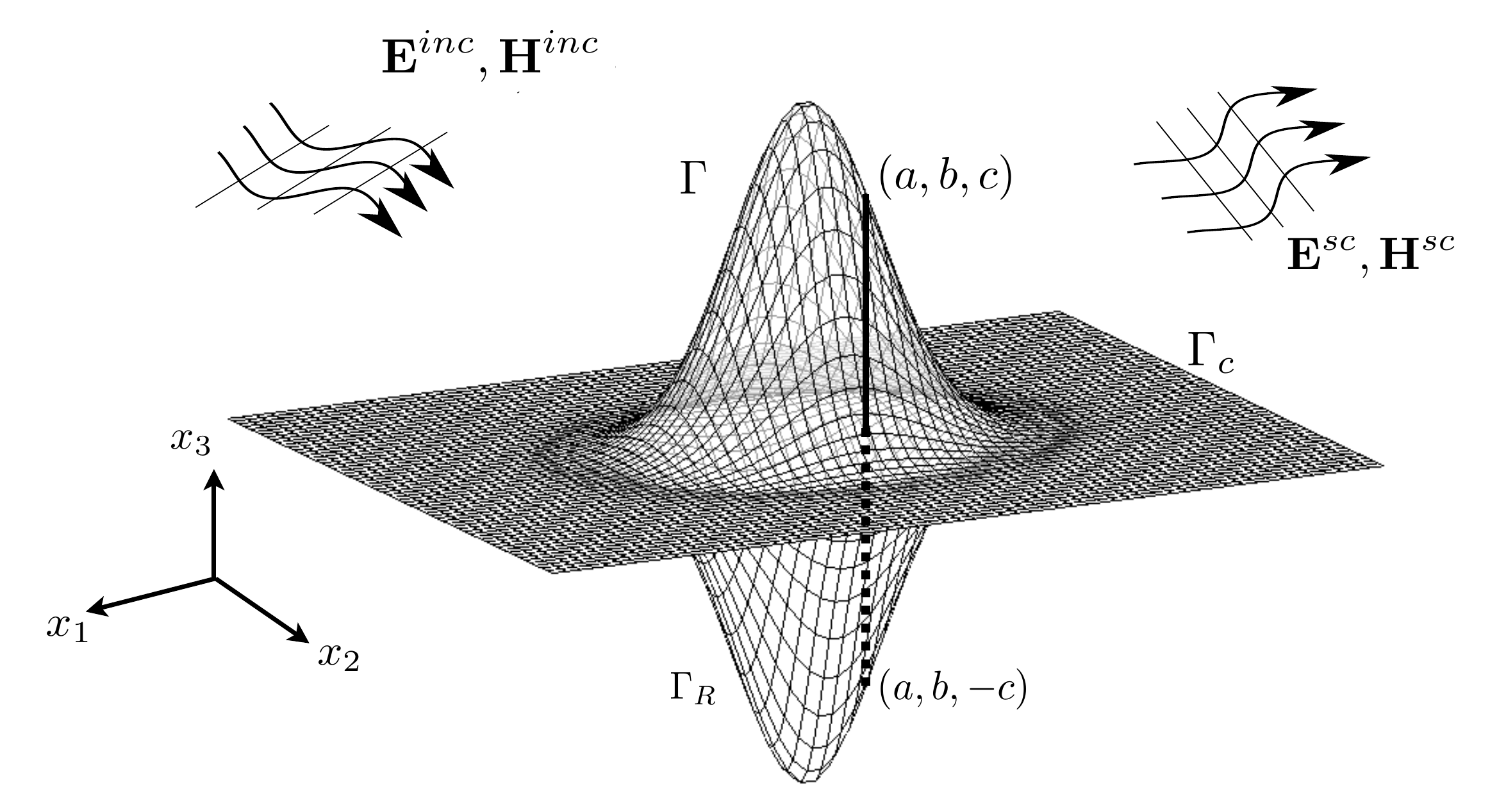}
\caption{A bump $\Gamma$ on a perfectly conducting
half-space $x_3 \geq 0$. $\Gamma_R$ denotes the reflection of the 
bump across the $x_1x_2$-plane. 
The unbounded half-space
boundary beyond $\Gamma$ is denoted by $\Gamma_c$.}\label{figure1a}
\end{figure}

\subsection{The cavity case}

Unfortunately, in the case of cavity deviations from a half-space, 
a more complicated representation
is required. In order to make use of the method of images,
following the approach of~\cite{JL2014}, we introduce an artificial
boundary $C_1$ that covers the cavity, as shown in 
Fig.~\ref{figure1}. The boundary $C_1$ must be sufficiently 
large so that its reflection
$C_2$ across the $x_1x_2$-plane does not intersect the cavity.
The domain $\mathbb{R}^{3+}\cup \Omega_1$ is now decomposed into two 
sub-domains. 

\begin{definition}
In a slight abuse of notation, we will continue to 
use $\Omega_1$ to denote the {\em interior} domain, bounded by
$\Gamma$, $B$ and $C_1$. 
We will refer to the upper half-space outside 
of $\Omega_1$ as the {\em exterior} domain.
\end{definition}

In the remainder of this paper, we will denote by 
$(\beint,\bhint)$ and $(\beext,\bhext)$ the scattered fields
in the interior and exterior domains, respectively.
The exterior scattered field 
$(\beext,\bhext)$ must satisfy Maxwell's equations~\eqref{maxequ}
in $\mathbb{R}^{3+}\backslash \Omega_1$, together with the 
boundary condition
\begin{equation}\label{halfspacebd}
\bn \times \beext = \bzero \quad \text{on } \Gamma^c
\end{equation}
and the interface (or transmission) conditions on $C_1$:
\begin{equation}
\begin{aligned}
\bn\times \beext &= \bn\times \beint, \\
\bn\times \bhext &= \bn\times \bhint.
\end{aligned}
\end{equation}

\begin{remark}
Since the domain $\Omega_1$ contains edges, the fields
$\beint$ and $\bhint$ in $\Omega_1$
are defined in the Sobolev space
\[ H(curl,\Omega_1)=\{\mathbf{u}\in (L^2(\Omega_1))^3, 
\nabla \times \mathbf{u} \in (L^2(\Omega_1))^3 \},
\]
where $(L^2(\Omega_1))^3$ denotes the space of component-wise 
square-integrable vector fields in $\Omega_1$. 
Moreover, the resulting integral equations we will derive
are consisted by Fredholm operators of index zero in the trace space of $H(curl,\Omega_1)$ ,
for which the Fredholm alternative still applies, i.e. uniqueness
implies existence. We will omit the proof but refer the reader to~\cite{BCS2002,KH2015} for a detailed
discussion. 
\end{remark}

As in the case of the bump, above, we can ensure 
that~\eqref{halfspacebd} is satisfied using the method of images.
In particular, we now represent the exterior fields by
\begin{equation}\label{eq_rep_ext}
\begin{aligned}
\beext &= -\frac{1}{ik}\nabla\times\nabla\times \cS_{C_1}^H \bj 
+\nabla \times \cS_{C_1}^H \bbm +\nabla \times \cS_{B} \bbm, \\
\bhext &= \frac{1}{ik} \nabla  \times \beext \\
&= \nabla\times \cS_{C_1}^H \bj +\frac{1}{ik}\nabla
\times\nabla 
\times \cS_{C_1}^H \bbm +\frac{1}{ik}\nabla \times\nabla \times 
\cS_{B} \bbm ,
\end{aligned}
\end{equation}
where image fields are given by
\begin{equation}\label{eq_sreflect}
\begin{aligned}
\nabla\times\nabla\times \cS_{C_1}^H \bj &= 
\nabla\times\nabla\times \cS_{C_1} \bj 
+ \nabla\times\nabla\times \cS_{C_2} \bj_R, \\
 \nabla \times \cS_{C_1}^H \bbm & = \nabla \times \cS_{C_1} \bbm 
+ \nabla \times \cS_{C_2} \bbm_R.
\end{aligned}
\end{equation}
The surface currents $\bj_R$ and $\bbm_R$ on $C_2$ are the images of 
the currents $\bj$ and $\bbm$ on $C_1$. More specifically, 
given $\bx=(x,y,z) \in C_1$ and its image point $\bx' = (x,y,-z)\in C_2$, 
the image currents are defined by
\begin{equation}\label{image}
\begin{aligned}
\bj{(\bx')} &= (-j_1,-j_2,j_3)  &\quad &\text{if }\bj{(\bx)} 
= (j_1,j_2,j_3),\\
\bbm{(\bx')} &= (m_1,m_2,-m_3)  & &\text{if }\bbm{(\bx)} = (m_1,m_2,m_3).
\end{aligned}
\end{equation}
It is straightforward to verify that~\eqref{halfspacebd} 
is enforced by symmetry. Note that the last term in 
the electric field representation 
in~\eqref{eq_rep_ext} does not involve an image current.
However, observe that $B$ is part of the half-space boundary. 
It is easy to check
that for any point beyond $C_1$, sources defined on $B$ make no 
contribution to the tangential
electric field on $\Gamma^c$.

\begin{remark}\label{remark3}
  We have introduced a magnetic current on $B$ even though $B$ is not
  part of the boundary of the exterior domain. This representation
  leads to a cancellation of hypersingular terms in the integral
  equation at the triple-junction where $C_1$, $C_2$, and $B$
  intersect, and hence to a bounded integral operator. This technique,
  which includes layer potentials on \emph{extra} boundary components to
  alter the kernels in the integral equation is sometimes called the
  \emph{global density} technique. For a more detailed discussion,
  see~\cite{GL2012,GHL2014,JL2014}.
\end{remark}

For the interior scattered fields, we let
\begin{equation}\label{eq_rep_int}
\begin{aligned}
\beint &= -\frac{1}{ik}\nabla\times\nabla\times \cS_{C_1}^H
   \bj + \nabla \times \cS_{C_1}^H \bbm + \nabla \times 
\cS_{B\cup \Gamma} \bbm, \\
\bhint &= \frac{1}{ik}  \nabla \times \beint \\
&= \nabla\times \cS_{C_1}^H \bj 
+\frac{1}{ik}\nabla \times\nabla \times \cS_{C_1}^H \bbm
                  +\frac{1}{ik}\nabla 
\times\nabla \times \cS_{B\cup \Gamma} \bbm.
\end{aligned}
\end{equation}
Note that the only difference in the interior representation when
compared with the exterior representation is that we have included a
contribution from a magnetic current on the cavity~$\Gamma$. Although $C_2$, the image surface of $C_1$, is not part of the boundary of $\Omega_1$, its contribution is added to the interior scattered fields by the similar reason as in Remark \ref{remark3}. 

For convenience, given two surfaces $\Gamma_s$ and $\Gamma_t$, and the
current $\bj$ on $\Gamma_s$, we define the following two surface
vector potentials for $\bx\in \Gamma_t$ by
\begin{align}
\mk_{\Gamma_t,\Gamma_s}\bj(\bx) 
&= \bn(\bx)\times \frac{1}{ik} \nabla\times\nabla\times \cS_{\Gamma_s}
\bj(\bx), \\
\mn_{\Gamma_t,\Gamma_s}\bj(\bx) &= \bn(\bx)\times \nabla\times
                                    \cS_{\Gamma_s}\bj(\bx).
\label{mnkdef}
\end{align}
When $\Gamma_s$ has a reflection, we define 
$\mk^H_{\Gamma_t,\Gamma_s}\bj$ and $\mn^H_{\Gamma_t,\Gamma_s}\bj$ 
to include the contribution from the reflected image currents as well,
as in~\eqref{eq_sreflect}.

When $\Gamma_t=\Gamma_s=\Gamma$, the integral operators 
$\mk$ and $\mn$ become singular. More precisely, the integral 
in $\mk_{\Gamma,\Gamma}$ is hypersingular and defined in the Hadamard 
principal-value sense, while $\mn_{\Gamma,\Gamma}$ is defined in the 
Cauchy principal-value sense~\cite{Cot2}. In this case, the following
jump relations hold:
\begin{align}
\lim_{\bx\rightarrow \Gamma^{\pm}} \bn(\bx)\times \frac{1}{ik} 
\nabla\times\nabla\times \cS_{\Gamma}\bj(\bx) &= 
\mk_{\Gamma,\Gamma}\bj(\bx), \label{jump1}\\ 
\lim_{\bx\rightarrow \Gamma^{\pm}} \bn(\bx)\times 
\nabla\times \cS_{\Gamma}\bj(\bx)&= \frac{1}{2}\bj(\bx) \pm 
\mn_{\Gamma,\Gamma}\bj(\bx), \label{jump2}
\end{align}
where $\Gamma^{\pm}$ denotes the side that corresponds to the
outward ($+$) or inward ($-$) normal, respectively.

Using the previous integral representations for the interior and
exterior fields, the boundary value problem
\begin{equation}
\begin{aligned}
\nabla \times \be - ik \, \bh &= 0, & &\text{in } \Omega_1 \cup
\mathbb R^{3+},\\
\nabla \times \bh + ik \, \be &= 0,& &\text{in } \Omega_1 \cup
\mathbb R^{3+},\\
\bn\times \beext &= \bn \times \beint, &\qquad &\text{on } C_1, \\
\bn\times \bhext  &= \bn \times \bhint,& &\text{on } C_1, \\
\bn\times \beint &=-\bn\times \beinc & &\text{on } B \cup \Gamma,
\end{aligned}
\end{equation} 
along with jump conditions~\eqref{jump1} and~\eqref{jump2},
immediately  yields a system of 
 integral equations for
$\bj$ and $\bbm$:
\begin{equation}\label{equsys}
\begin{aligned}
\bbm - \mn_{C_1,\Gamma}\bbm  &= \bzero & \qquad&\text{on } C_1, \\
\bj - \mk_{C_1,\Gamma}\bbm &=  \bzero &  &\text{on } C_1, \\
\frac{1}{2}\bbm + \mn_{B, \Gamma}\bbm &= \bzero & &\text{on } B, \\
\frac{1}{2}\bbm - \mk^H_{\Gamma, C_1}\bj+
\mn^H_{\Gamma, C_1}\bbm+ \mn_{\Gamma, B\cup \Gamma}\bbm &=
-\bn\times \beinc 
& &\text{on } \Gamma.
\end{aligned}
\end{equation}
Due to the existence of corners, the integral system \eqref{equsys} is not second kind Fredholm equation. Nevertheless, the system is well conditioned once the corner singularity is well resolved numerically.

\begin{remark}
Given the magnetic current $\bbm$ along $\Gamma$, 
the first three equations the system~\eqref{equsys} are explicitly 
solvable because they only involve the application of an off-surface
layer potential.
Therefore, our formulation can be 
reduced to an unknown magnetic current~$\bbm$ along~$\Gamma$ only. 
In particular, by substitition of the first three equations
in~\eqref{equsys} into the fourth equation in~\eqref{equsys},
we have the integral equation
\begin{equation} \label{equonlyM}
\frac{1}{2}\bbm - \mk^H_{\Gamma,
  C_1}\mk_{C_1,\Gamma}\bbm+
\mn^H_{\Gamma, C_1}\mn_{C_1,\Gamma}\bbm - 2 
\mn_{\Gamma,B}\mn_{B, \Gamma}\bbm
+ \mn_{\Gamma, \Gamma}\bbm= -\bn\times \beinc
\end{equation}
for $\bbm$ along $\Gamma$.
Not only does this observation reduce the number of unknowns
significantly, it will play a key role in avoiding low-frequency
breakdown in the electric field, as shown in 
Section~\ref{sec_lowfreq}.
\end{remark}

\section{Existence and uniquness}
\label{sec_prof}

Since the integral system~\eqref{equsys} consists of Fredholm
operators of index zero, by the Fredholm alternative, existence
follows from uniqueness.
This is given by the following theorem.

\begin{theorem}
Equation~\eqref{equsys} admits a unique solution $\bj$ and $\bbm$ for $k>0$. 
\end{theorem}

\begin{proof} 
It suffices to show that the system~\eqref{equsys} has only the
trivial solution for $k>0$.
For this, let us denote by
$\Omega_2$ the region bounded by $C_2$, $\Gamma$ and $B$. 
The proof involves three steps.

First, for $\bj$ and $\bbm$ given in~\eqref{equsys}, define the induced
electromagnetic field $(\be,\bh)$ 
for $\bx \in \mathbb{R}^3\backslash(\Omega_1\cup \Omega_2)$ by
\begin{equation}\label{defrep1}
\begin{aligned}
  \be &= -\frac{1}{ik}\nabla\times\nabla\times \cS_{C_1}^H\bj +\nabla
  \times \cS_{C_1}^H\bbm +\nabla \times \cS_{B}\bbm, \\
  \bh &= \nabla\times \cS_{C_1}^H\bj +\frac{1}{ik}\nabla \times\nabla
  \times \cS_{C_1}^H\bbm +\frac{1}{ik}\nabla \times\nabla \times
  \cS_{B}\bbm,
\end{aligned}
\end{equation}
 and for $\bx\in \Omega_1\cup \Omega_2$,
\begin{equation}\label{defrep2}
\begin{aligned}
\be &= -\frac{1}{ik}\nabla\times\nabla\times \cS_{C_1}^H\bj 
+ \nabla \times \cS_{C_1}^H\bbm + \nabla \times \cS_{B\cup \Gamma}\bbm,   \\
\bh &= \nabla\times \cS_{C_1}^H\bj +\frac{1}{ik}\nabla \times\nabla
\times 
\cS_{C_1}^H\bbm +\frac{1}{ik}\nabla \times\nabla \times \cS_{B\cup \Gamma}\bbm.
\end{aligned}
\end{equation}
From~\eqref{equsys}, we have that
\begin{equation}
\lim_{{\substack{\bx\rightarrow \Gamma^c\\ 
\bx\in \mathbb{R}^3\backslash(\Omega_1\cup 
\Omega_2)}}} \bn \times \be = 0, \qquad 
\lim_{\substack{\bx\rightarrow B\cup \Gamma\\\bx \in \Omega_1}} 
\bn \times \be = 0 
\end{equation}
and 
\begin{equation}\label{concond}
[\bn\times\be] = 0, \quad [\bn\times \bh] = 0 \quad \text{on } C_1,
\end{equation}
where $[\cdot]$ denotes the jump in the corresponding the field.

Let $D\subset \mathbb{R}^{3+}$ be a finite region that is bounded 
by a sufficiently large hemisphere $\partial D$
 and the surfaces $\Gamma^c$ and $C_1$. 
Applying Green's first vector identity~\cite{Cot2} to $\be$ and its complex 
conjugate $\overline{\be}$ in $D$, and recalling the fact that 
$\nabla\cdot \be =0$ in this region yields 
\begin{equation} \label{green1}
\int_{\partial D \cup C_1} (\bn \times \be) \cdot (\nabla \times
\overline{\be})
\, dA = \iint_D \left( |\nabla\times \be|^2-k^2|\be|^2 \right) \, dV .
\end{equation}
Similarly, applying the same identity to $\be$ in $\Omega_1$, we have
\begin{equation}
\label{green2}
\int_{ C_1\cup \Gamma} (\bn \times \be) \cdot (\nabla \times
\overline{\be})
 \, dA
= \iint_{\Omega_1}\left( |\nabla\times \be|^2-k^2|\be|^2 \right) \, dV.
\end{equation} 
Note that $\bn$ in~\eqref{green1} and~\eqref{green2} denotes the 
exterior normal with respect to the regions $D$ and $\Omega_1$, and
that the right hand side of both relations is purely real (for real $k$).
Combining~\eqref{green1} and~\eqref{green2}, using the continuity
condition~\eqref{concond} along $C_1$, and taking imaginary parts 
we have
\begin{equation}\label{equimag}
\Im{\int_{\partial D} \left( \bn\times \be \right) \cdot 
\left( \nabla \times \overline{\be}   \right)} = \Im{\int_{\Gamma} 
\left( \bn\times \be \right) \cdot \left( \nabla \times \overline{\be}
\right) } = 0, 
\end{equation}
where $\Im$ denotes the imaginary part. Lastly, since $\bn \times \be
= 0$  on $\Gamma$, relation~\eqref{equimag}
implies that the electric field
 is identically zero in $D$ by Rellich's lemma~\cite{Cot2}. 
By analytic continuation from $C_1$ into $\Omega_1$,
we also obtain $\be=0$ and $\bh=0$ in $\Omega_1$.

The second step in the proof is to show 
that the fields $(\be,\bh)$ in $\Omega_2$ satisfy 
the following boundary conditions
\begin{equation}\label{cc}
\begin{aligned}
\bn \times \be & = \bbm & \quad &\text{on } \Gamma \cup B, \\
\bn \times \bh & = \mathbf{0} & & \text{on } \Gamma\cup B, \\
\bn \times \be & = \mn_{C_2,\Gamma'}\bbm+\mn_{C_2,\Gamma}\bbm & &
 \text{on } C_2, \\
\bn\times \bh & = \mk_{C_2,\Gamma'}\bbm+\mk_{C_2,\Gamma}\bbm & & 
 \text{on } C_2  .
\end{aligned}
\end{equation}
Here, just to clarify, $\bn$ on $\Gamma \cup B$ is the exterior normal
with respect to $\Omega_1$ and $\bn$ on $C_2$ is the exterior normal
with respect to $\Omega_2$. The surface $\Gamma'$ 
is the image of $\Gamma$
with respect to the $x_1x_2$-plane, and $\bbm$ on $\Gamma'$ is the
corresponding image magnetic
current. By~\eqref{defrep1},~\eqref{defrep2}, and the jump
relations~\eqref{jump1},~\eqref{jump2}, we have that
\begin{equation}
[\bn\times \be] = -\bbm \quad \text{and} \quad 
[\bn \times \bh ] = 0  \quad  \text{on } \Gamma \cup B .
\end{equation}
Using the fact that $\be =\mathbf{0}$ and $\bh=\mathbf{0}$ 
in $\Omega_1$, we obtain the 
desired boundary conditions~\eqref{cc} 
on $\Gamma\cup B$ for the fields $(\be,\bh)$ in $\Omega_2$.

Turning to the boundary $C_2$, we have the jump conditions
\begin{equation}\label{c12}
[\bn\times \be] = -\bbm' + \mn_{C_2,\Gamma}\bbm \quad \text{and} \quad
[\bn\times \bh] = -\bj' +  \mk_{C_2,\Gamma}\bbm \quad \text{on } C_2. 
\end{equation} 
Here, $\bj'$ and $\bbm'$ are the image surface currents on $C_2$.
By equation~\eqref{equsys} and symmetry,
\begin{align}\label{c34}
\bbm' = \mn_{C_2, \Gamma'}\bbm' \quad \text{and} \quad 
\bj' = \mk_{C_2, \Gamma'}\bbm' \quad \text{on } C_2,
\end{align}
where $\bbm'$ on $\Gamma'$ is again the image magnetic current 
of $\bbm$ on $\Gamma$.
Combining~\eqref{c12} and~\eqref{c34}, we obtain the remaining 
boundary conditions in~\eqref{cc}. 

The third, and final, step in the proof is to show that
the electromagnetic field satisfying~\eqref{cc} is 
identically zero. 
For this, extend $\be$ and $\bh$ from $\Omega_2$ to 
$\mathbb{R}^{3-}\backslash\Omega_2$ by letting
\begin{equation}
\begin{aligned}
\be & = -\nabla\times \cS_{\Gamma'}\bbm' + \nabla\times \cS_{\Gamma}\bbm  \\
\bh & = -\frac{1}{ik}\nabla\times\nabla\times \cS_{\Gamma'}\bbm'+
\frac{1}{ik}\nabla\times\nabla\times \cS_{\Gamma}\bbm .
\end{aligned}
\end{equation}
It is easy to see that $(\be,\bh)$ 
is an electromagnetic field satisfying
\begin{equation}  
\bn\times \bh = \bzero \quad \text{on } \Gamma^c .
\end{equation} 
Following the same argument as in the first step 
(with $\be$ and $\bh$ exchanged), 
we must have $\be =\mathbf{0}$ and $\bh =\mathbf{0}$ 
in $\Omega_2$, which implies $\bbm=\mathbf{0}$ on $\Gamma$. Therefore,
by~\eqref{equsys}, we obtain $\bj = \mathbf{0}$ and $\bbm=\mathbf{0}$ 
on $C_1$.\\
\end{proof}

A similar proof can be obtained in the case where $\Re (k)>0$ and
$\Im (k)>0$. In the static case where $k=0$, our representation is not valid and must be altered. However, the existence and uniqueness can be handled directly by electro- and magneto-static arguments.

\section{Low-frequency breakdown}
\label{sec_lowfreq}

It is clear that the integral
representations~\eqref{eq_rep_ext} and~\eqref{eq_rep_int}
are numerically unstable as $k\rightarrow 0 $ due
to the explicit $1/ik$ scaling.
This problem is not intrinsic to the Maxwell system~\eqref{maxequ},
where the electric and magnetic field simply uncouple in the static
limit.  Rather, it is due to the use of vector currents $\bj$, $\bbm$
as the unknowns. 
The resulting loss of precision is generally referred to as
\emph{low-frequency breakdown}~\cite{ZC2000}.  Rather than develop a
new mathematical formalism to overcome this, as in \cite{EG10,EG13},
we modify our method described above to create a representation that
is stable as $k\rightarrow 0 $.  In short, using properties of the
electromagnetic fields and vector identities, we are able to express
the electric field only in terms of $\bbm$, and in the process,
eliminate the terms which scale as $1/ik$.
The magnetic field representation still formally suffers from
low-frequency breakdown, but numerically it is relatively benign as
$k\to 0$.

In order to express $\be$ only in terms of well-scaled operators of
$\bbm$, we require the following identity.

\begin{lemma}\label{lemma1}
Let $\Gamma$ be an open surface with smooth boundary $\ell$, and let
$\bj$ be a smooth tangential vector field along~$\Gamma$. Then
\begin{equation}
\frac{1}{ik}\nabla\times\nabla\times \cS_{\Gamma}\bj 
= -ik\cS_{\Gamma}\bj +\nabla \cS_{\Gamma}
\left(\frac{\nabla_{\Gamma}\cdot \bj }{ik}\right) 
-\nabla \cS_{\ell}\left(\frac{\bj\cdot \bbb}{ik}\right)
\end{equation}
where $\bbb = \btau \times \bn$ is the 
outward bi-normal vector along $\ell$, with $\btau$ the unit 
tangent vector along~$\ell$ and~$\bn$ the surface normal (oriented so
that $\bbb$ points away from the surface).
\end{lemma} 

\begin{proof}  
Using the point-wise vector identity 
$\nabla\times \nabla \times {\bf a}= - \Delta {\bf a}+
\nabla \nabla\cdot {\bf a}$, we have
\begin{equation}\label{vectident1}
\frac{1}{ik}\nabla\times\nabla\times \cS_{\Gamma}\bj 
= -\frac{1}{ik} \Delta  \cS_{\Gamma}\bj + \nabla \nabla\cdot 
\cS_{\Gamma}\left(\frac{\bj}{ik}\right).
\end{equation}
Since $\cS_{\Gamma}\bj$ is a vector-valued Helmholtz potential, 
it is clear that
\begin{equation}
-\frac{1}{ik} \Delta  \cS_{\Gamma}\bj = -ik\cS_{\Gamma}\bj.
\end{equation}
For the second term on the right-hand side of~\eqref{vectident1}, 
we apply Stokes's identity~\cite{Ned01} on the surface $\Gamma$:
\begin{equation}
\begin{aligned}
\nabla\cdot \cS_{\Gamma}\left(\frac{\bj}{ik}\right) & = 
-\int_{\Gamma}\nabla_{\by,\Gamma}\cdot\left(G(\bx,\by) \, 
\frac{\bj(\by)}{ik} \right) \, dA_y + \int_{\Gamma}G(\bx,\by) \, 
\frac{ \nabla_\Gamma \cdot 
\bj(\by)}{ik}  \, dA_y\\
& =  - \cS_{\ell}\left(\frac{\bj\cdot \bbb}{ik}\right)
+ \cS_{\Gamma} \left(\frac{\nabla_{\Gamma}\cdot \bj}{ik}\right),
\end{aligned}
\end{equation}
where $\nabla_{\by,\Gamma} \cdot$ denotes the surface divergence 
with respect to the variable $\by$.
\end{proof}

\begin{corollary}\label{corollary1}
Let $\Gamma$ be a surface whose boundary $\ell$ is a curve that
lies on the plane 
$x_3=0$, and let~$\bj$ and~$\bbm$ be surface electric and magnetic 
currents on~$\Gamma$ with image currents defined
in~\eqref{image}. Note that the image currents also lie along $x_3 =
0$. Then
\begin{equation}
\begin{aligned}
\frac{1}{ik}\nabla\times\nabla\times \cS^H_{\Gamma}\bj 
&= -ik\cS^H_{\Gamma}\bj +\nabla \cS^H_{\Gamma}
\left(\frac{\nabla_{\Gamma}\cdot \bj}{ik}\right), \\
\frac{1}{ik}\nabla\times\nabla\times \cS^H_{\Gamma}\bbm &= -ik
\cS^H_{\Gamma}\bbm +\nabla \cS^H_{\Gamma}\left(\frac{\nabla_{\Gamma}
\cdot \bbm}{ik}\right)
-2\nabla \cS_{\ell}\left(\frac{\bbm\cdot \bbb}{ik}\right).
\end{aligned}
\end{equation}

\end{corollary} 

We can now establish the following identity:

\begin{lemma}\label{lfident}
For a tangential electric and magnetic currents $\bj$, $\bbm$ 
along $\Gamma$, defined as in~\eqref{eq_rep_int},
\begin{equation}\label{identity2}
\frac{\nabla_{C_1}\cdot \bj}{ik} = -\bn\cdot \nabla\times \cS_{\Gamma}\bbm .
\end{equation} 
\end{lemma}

\begin{proof}
First, apply Corollary~\ref{corollary1} to rewrite the representation of 
the electric field in the form
\begin{equation}\label{eq_rep_2}
\begin{aligned}
\beext &= ik\cS^H_{C_1}\bj - \nabla
         \cS^H_{C_1}\left(\frac{\nabla_{C_1}\cdot \bj}{ik}\right)
         +\nabla \times \cS_{C_1}^H \bbm +\nabla \times \cS_{B} \bbm, \\
\beint &= ik\cS^H_{C_1}\bj - \nabla \cS^H_{C_1}
\left(\frac{\nabla_{C_1}\cdot \bj}{ik}\right) + \nabla \times
         \cS_{C_1}^H \bbm 
+ \nabla \times \cS_{B\cup \Gamma} \bbm.
\end{aligned}
\end{equation}
Taking the difference of~$\beext$ and~$\beint$ in~\eqref{eq_rep_2}, 
computing normal components, using the
jump relations for the single layer potential, and recalling 
the  continuity condition
\begin{equation}
\bn\cdot \beext = \bn\cdot \beint \quad \text{on } C_1,
\end{equation}
 yields the desired result in~\eqref{identity2}.
\end{proof}

We can now derive low-frequency versions of integral
equation~\eqref{equonlyM} and representations for the electric and
magnetic fields.

\subsection{A modified integral equation}

Using the previous results, we can now derive an integral equation
along~$\Gamma$ which does not suffer from low-frequency breakdown,
as does equation~\eqref{equonlyM}.
Inspection of the various terms in~\eqref{equonlyM}
shows that $1/ik$ scaling is present only in the term 
\begin{equation}
 -\mk^H_{\Gamma, C_1}\mk_{C_1,\Gamma}\bbm.
\end{equation}
By application of Lemma~\ref{lemma1} and use 
of the result in Corollary~\ref{corollary1}, 
this term can be replaced as
\begin{equation}\label{equlowfreqterm}
 -\mk^H_{\Gamma, C_1}\mk_{C_1,\Gamma}\bbm = 
\bn\times ik \cS^H_{C_1}\bj -\bn\times \nabla \cS^H_{C_1}
\left(\frac{\nabla_{C_1}\cdot \bj}{ik}\right),
\end{equation}
where, recalling that by equation~\eqref{equsys}
\begin{equation}\label{eq_jinm}
\begin{aligned}
\bj = \mk_{C_1,\Gamma}\bbm = \bn\times\left(\frac{1}{ik} 
\nabla\times\nabla\times \cS_{\Gamma}\bbm\right).
\end{aligned}
\end{equation}
Using~\eqref{eq_jinm}, 
the first term in~\eqref{equlowfreqterm} can be rewritten in terms of 
$\bbm$ on $\Gamma$:
\begin{equation} \label{identity1}
\bn \times ikS^H_{C_1}\bj 
= \bn \times \cS_{C_1}^H \left( 
\bn\times \nabla\times\nabla\times \cS_{\Gamma}\bbm \right).
\end{equation}
Combining~\eqref{identity1} and~\eqref{identity2}, we can 
have the modified equation along $\Gamma$:
\begin{multline}\label{nobreakdown}
\frac{1}{2}\bbm+\bn\times \cS_{C_1}^H\left( \bn \times
\nabla\times\nabla\times \cS_{\Gamma}\bbm \right) 
+ \bn\times \nabla \cS^H_{C_1}\left( \bn\cdot \nabla\times
\cS_{\Gamma}\bbm \right)  \\
+\mn^H_{\Gamma, C_1} \mn_{C_1,\Gamma} \bbm  - 2 \mn_{\Gamma, B}
\mn_{B, \Gamma}\bbm + \mn_{\Gamma, \Gamma}\bbm= -\bn\times 
\beinc.
\end{multline}
This is equivalent to~\eqref{equonlyM}, but is clearly immune from 
low-frequency breakdown.

\subsection{Field calculations}

Furthermore, given the solution~$\bbm$ to integral
equation~\eqref{nobreakdown}, we have the following representations
for the electric field which do not suffer from low-frequency
breakdown:
\begin{equation}\label{eq_e_lowfreq}
\begin{aligned}
  \beext &= ik \cS^H_{C_1} \mk_{C_1,\Gamma} \bbm - \nabla \cS^H_{C_1}\left( \bn\cdot
    \nabla\times \cS_{\Gamma}\bbm \right)
  +\nabla \times \cS_{C_1}^H \bbm +\nabla \times \cS_{B} \bbm, \\
  \beint &= i k  \cS^H_{C_1} \mk_{C_1,\Gamma}
    \bbm  - \nabla \cS^H_{C_1}\left( \bn\cdot
    \nabla\times \cS_{\Gamma}\bbm \right) +\nabla \times \cS_{C_1}^H
  \bbm +\nabla \times \cS_{B\cup \Gamma} \bbm.
\end{aligned}
\end{equation}
Once $\bbm$ along $\Gamma$ 
has been obtained by solving~\eqref{nobreakdown}, we can
obtain $\bbm$ on $C_1\cup B$ through~\eqref{equsys}.
Evaluation via representation~\eqref{eq_e_lowfreq} is stable as
$k\to 0$.
On the other hand, we are left with computing the magnetic field as
\begin{equation}
\bhext = \frac{1}{ik} \nabla \times 
\beext, \qquad \bhint = \frac{1}{ik} \nabla \times \beint,
\end{equation}
which are inherently first-order operations in the variable~$\bbm$ and
will obviously suffer from low-frequency breakdown.
As an alternative, we may try to rewrite $\bh$ using vector identities
and in terms of the variable $\bj$. In the case of $\bhint$, for
example, we have
\begin{equation}
\bhint = \nabla\times \cS_{C_1}^H \bj 
+\frac{1}{ik}\nabla \times\nabla \times \cS_{C_1}^H \bbm
                  +\frac{1}{ik}\nabla 
\times\nabla \times \cS_{B\cup \Gamma} \bbm.
\end{equation}
The expression for $\bhext$ is nearly identical.
Any attempt using the previous lemmas or corollary to simplify this
representation will require the (numerical) evaluation of the quantities
$\bj$, $\nabla_\Gamma \cdot \bbm/ik$, and $\bbm \cdot \bbb /ik$.
Empirically, these expressions can be evaluated at the cost of a mild
loss of numerical precision. 
The quantity~$\bj$ can be directly evaluated via $\bj =
\mk_{C_1,\Gamma} \bbm$, or by using the identity:
\begin{equation}\label{Jcurrent}
\bj = \bn\times\left( 
-ik\cS_{\Gamma}\bbm +\nabla \cS_{\Gamma}\left( \frac{\nabla_{\Gamma}\cdot
  \bbm}{ik} \right)
 -\nabla \cS_{\ell}\left ( \frac{\bbm\cdot \bbb}{ik}\right) \right).
\end{equation}
We compute the function $\rho_M = \nabla_\Gamma \cdot \bbm/ik$ merely by
spectral differentiation in a $10^{\text{th}}$-order Legendre
discretization, and the term~$\phi_M = \bbm \cdot \bbb /ik$ by
extrapolation.

These schemes lead to roughly an~$\mathcal O(\log k)$ loss in absolute
precision. For example, as discussed in the numerical examples
section, for $k \approx 10^{-10}$, we are able to obtain approximately
6 digits of accuracy.

\begin{remark}
  The function $\phi_M$ cannot be evaluated naively for
  small~$k$. Numerical experiments indicate that
  ${\bbm\cdot \bbb} \approx \mathcal O(k)$ as $k \rightarrow 0$, but
  we have not found a proof of this fact. Similar to $\rho_M$, we
  numerically found that $\phi_M$ has a low-frequency limit. In
  particular, we can formally expand:
\begin{equation}
  \frac{\bbm \cdot \bbb}{ik} = m_1 + km_2 + \ldots.
\end{equation}
The quantity $\phi_M$ can then be evaluated for several distinct
non-zero values of $k$, and the coefficients $m_j$ can be estimated
to the desired order of accuracy. These estimated values can then be
used to compute $\phi_M$ for any $k\geq 0$. The result can then be
used in~\eqref{Jcurrent} to evaluate the $\bh$ field.  This form of
low-frequency breakdown is, therefore, much less pernicious than that
addressed by Lemma~\ref{lfident}, and that present in the evaluation
of $\rho_M$, which is sometimes referred to as
dense-mesh breakdown. Nevertheless, we consider it to be an open
problem to find an integral formulation which avoids the need for this
asymptotic approach.
\end{remark}

\section{Separation of variables for boundary 
integral operators}
\label{sec_fourier}

For arbitrarily shaped cavities, a full three-dimensional treatment of
quadratures and geometry discretization
is required to evaluate the integral operators discussed in the
previous section, not to mention schemes to
solve the corresponding integral equations.
While fast multipole methods reduce the computational complexity
of applying such integral operators to 
$\mathcal O(N)$ or $\mathcal O(N \log N)$, 
(see for example,~\cite{Cheng2006}) and high-order
quadratures have been developed for weakly-singular kernels 
on arbitrary surface triangulations~\cite{Bremer2012},
the associated constants implicit in the $\mathcal O(\cdot)$ scaling
are relatively large.
On the other hand, for a wide class of geometries -- namely those with
rotational symmetry -- two useful accelerations are easily obtained.
First, in problems for which the Green's function is translation
invariant,
one can apply separation of variables in the azimuthal angle,
$\theta$, relative to the $x_3$-axis, 
and then Fourier decompose the problem.
This transforms the original integral equation, defined on a surface
in three dimensions, into a 
sequence of uncoupled integral equations (one for each Fourier mode)
defined along a one-dimensional \emph{generating curve}. 
Second, the resulting linear systems are much smaller,
 and the associated quadrature issues
are much easier to 
address~\cite{HK2014,Hao2015,YHM2012,Liu2015,gedney_1990,oneil_2016}.
We avoid a detailed description of axisymmetric integral equation
solvers here, and instead point the reader to the previous references
for discussions related to quadrature and kernel evaluation.
In what follows below, we provide a brief description of the
discretization relevant to our cavity scattering problem.

\begin{figure}
\center
\includegraphics[width=.7\linewidth]{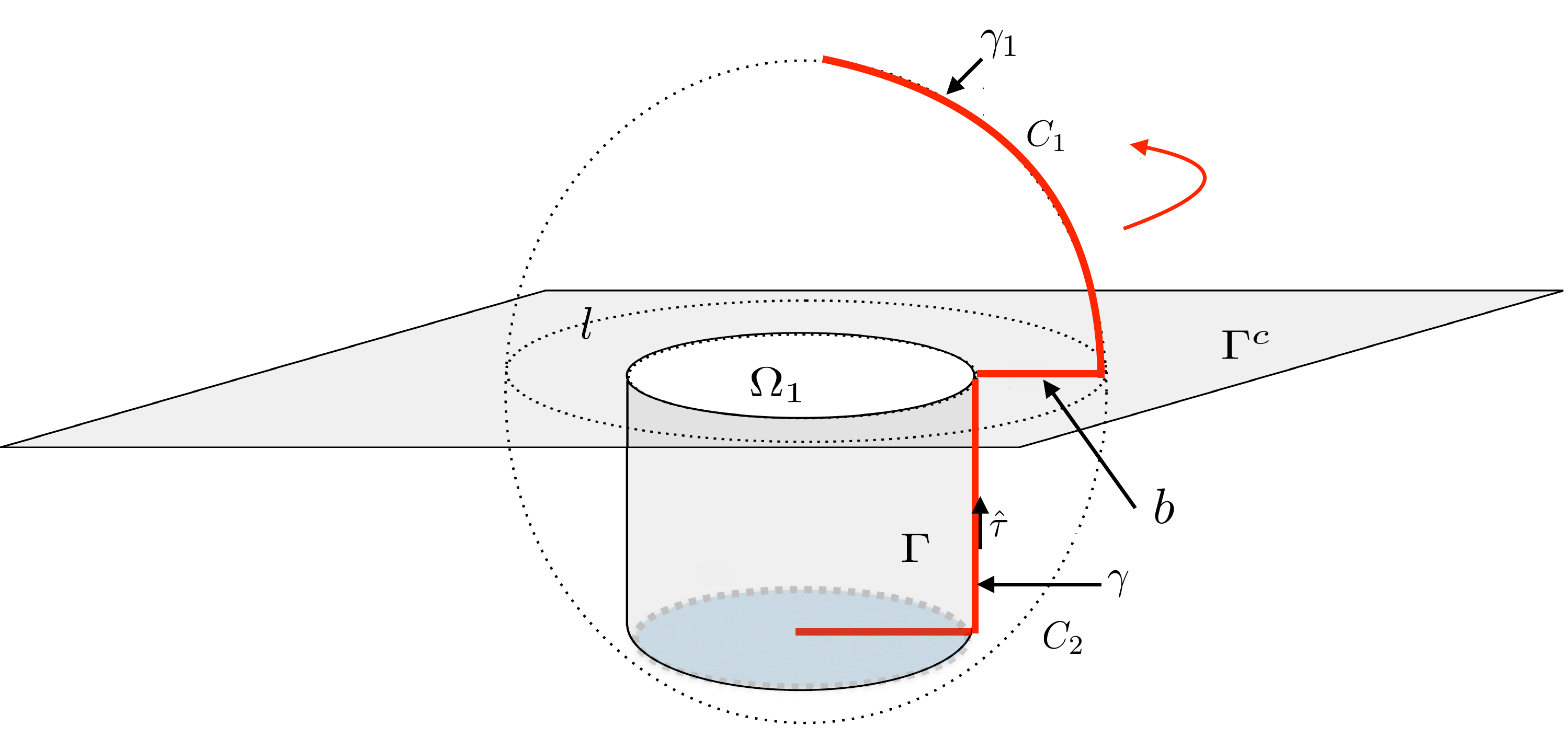}
\caption{Axisymmetric cavity with generating curve $g = 
\gamma\cup b \cup \gamma_1$. }\label{figure2}
\end{figure}

\subsection{Modal kernels and operators}
\label{sec_operators}

A simple example of the geometrical setup is shown in Fig.~\ref{figure2}.
A point $\bx=(x_1,x_2,x_3)$ in $\mathbb R^3$ will be denoted in
the usual cylindrical coordinates as $\bx=(r,\theta,z)$.
The standard unit vectors in cylindrical 
coordinates will be denoted as $(\er,\et,\ez)$.
The generating curve is given by
\mbox{$g= \gamma\cup b \cup \gamma_1$}, which we assume is 
parameterized 
by $g(s) = (r(s),z(s))$, where $s$ denotes arclength.
The tangent vector along the generating curve 
is then $\bt= \frac{dr}{ds} \er + \frac{dz}{ds} \ez$, and 
the exterior normal $\bn$ is given by \mbox{$\bn = \frac{dz}{ds}  \er - 
\frac{dr}{ds} \ez$}. 
Relative to the surface frame $\bt$, $\bn$, $\et$, any tangential
vector field $\bj$ along $\Gamma \cup B \cup C_1$ can be written as:
\begin{equation}\label{eq_jsurf}
\bj(s,\theta) = \sum_m (J_m^1(s) \, \bt + J_m^2(s) \, \et) \, e^{im\theta} .
\end{equation}
In the scalar case, a second-kind boundary integral equation 
on a body of revolution $\Omega$,
\begin{equation}
u(\bx) + \int_{\partial\Omega_1} G(\bx-\bx') \, u(\bx') \, dA(\bx') = f(\bx),
\end{equation}
can immediately be decomposed into a sequence of
decoupled equations
\begin{equation}\label{scalar1}
u_m(\bx) + 2\pi \int_g G_m(r,z,r',z') \, u_m(r',z') \, r' \, 
ds(r',z') = f_m(r,z)
\end{equation}
where
\begin{equation}\label{scalar2}
\begin{aligned}
u(\bx) &= \sum_m u_m(r,z) \, e^{im\theta}, \\
f(\bx) &= \sum_m f_m(r,z) \, e^{im\theta}, \\
G(\bx-\bx') &= \sum_m G_m(r,z,r',z') \, e^{im(\theta-\theta')}.
\end{aligned}
\end{equation}
However, in the vector-valued integral equation setting, 
the components of the unknown $\bj$ do not fully 
separate when expressed in terms of
local, tangential coordinates~\cite{oneil_2016}.
We therefore need to compute the action of the single-layer potential
operator $\cS$, and its derivatives, on a tangential density $\bj$.
Using the fact that relative to the Cartesian unit vectors $\exx$,
$\eyy$, $\ezz$:
\begin{equation}
\begin{aligned}
\er &= \cos\theta \, \exx + \sin\theta \, \eyy, \\
\et &= -\sin\theta \, \exx + \cos\theta \, \eyy, \\
\ez &=  \ezz,
\end{aligned}
\end{equation}
it is straightforward to verify that if $\bj$ is as
in~\eqref{eq_jsurf}, then
\begin{equation}\label{eq_sj}
\cS_{\Gamma}\bj = \sum_m \left( c_m^1 \er + c_m^2 \et + c_m^3\ez
\right) \, e^{im\theta}
\end{equation}
where 
\begin{equation}
\begin{aligned}
c_m^1 &=  \int_{\gamma} \gccm (r,z,r',z') \, J^1_m(s) \, \frac{dr'}{ds}
 \, r' \, ds -i \int_{\gamma} \gssm(r,z,r',z') \, J^2_m(s) \, r' \, ds, \\
c_m^2 &= i\int_{\gamma} \gssm(r,z,r',z') \, J^1_m(s) \frac{dr'}{ds} \, 
r' \, ds 
+\int_{\gamma} \gccm(r,z,r',z') \, J^2_m(s) \, r' \, ds, \\
c_m^3  &= \int_{\gamma}  \gm (r,z,r'z') \, J^1_m(s) \frac{dz'}{ds} \, 
r' \, ds,
\end{aligned}
\end{equation}
and the kernels $\gm$, $\gccm$, and $\gssm$ are given by:
\begin{equation}\label{eq_green123}
\begin{aligned}
\gm(r,z,r',z') &= \frac{1}{2\pi} 
\int_{0}^{2\pi} \frac{e^{ikR}}{4\pi R} \, e^{-im\phi} \, 
d\phi , \\
\gccm(r,z,r',z') &= \frac{1}{2\pi} 
\int_{0}^{2\pi} \frac{e^{ikR}}{4\pi R} \, \cos m\phi \,
\cos\phi \,  d\phi,  \\
\gssm(r,z,r',z') &= \frac{1}{2\pi} 
\int_{0}^{2\pi} \frac{e^{ikR}}{4\pi R} \, \sin m\phi \,
\sin\phi \, d\phi,
\end{aligned}
\end{equation}
with $R$ given by the Euclidian distance in cylindrical coordinates:
\begin{equation}
R = \sqrt{r^2+r'^2 - 2 rr'\cos\phi + (z-z')^2},
\end{equation}
with $\phi = \theta-\theta'$ denoting the azimuthal angle between
$\bx$ and $\bx'$.
It is understood in the previous formulas that for \emph{source}
points $(r',z')$ on the boundary, evaluation is in terms of
the parameterization of the curve, i.e.: $r' = r(s')$ 
and $z' = z(s')$ for
some $s'$.
Expressions for all other boundary operators, for example $\mathcal K$ 
and $\mathcal N$, can be obtained from the above expressions by taking
derivatives with respect to sources and targets.

In particular, 
formulas for the curl and divergence of $\cS \bj$ can be calculated
immediately from application of these operators in cylindrical
coordinates to expression~\eqref{eq_sj}, with the partial derivatives
being taken directly on the kernel functions. Differentiation with
respect to $\theta$ is diagonal. For this reason, 
we omit these formulas here.

However, it is useful to provide an expression for 
a less common differential operator, namely $\nabla \nabla \cdot$,
used when applying $\nabla \times \nabla \times$ to Helmholtz
potentials, see Lemma~\ref{lemma1}. The operator~$\mk$ involves this.
To this end,  we have:
\begin{multline}
  \nabla \nabla \cdot \cS_{\Gamma}\bj
  = \left( -\frac{c_m^1}{r^2}+\frac{1}{r}\frac{\partial
      c_m^1}{\partial r}
+\frac{\partial^2c_m^1}{\partial r^2} 
-\frac{im}{r^2}c_m^2 +\frac{im}{r}\frac{\partial c_m^2}{\partial r} 
+ \frac{\partial^2c_m^3}{\partial r\partial z}\right)\er \\
  +\frac{im}{r}\left(\frac{c_m^1}{r} + \frac{\partial
    c_m^1}{\partial r} +\frac{im}{r}c_m^2 +
  \frac{\partial c_m^3}{\partial z}\right) \et
  +\bigg(\frac{1}{r}\frac{\partial c_m^1}{\partial z} +
  \frac{\partial^2c_m^1}{\partial z \partial r}
  +\frac{im}{r}\frac{\partial c_m^2}{\partial z} +
  \frac{\partial^2 c_m^3}{\partial z^2}\bigg)\ez.
\end{multline}
Furthermore, in order to discretize~\eqref{nobreakdown}, the modified
integral equation free from low-frequency breakdown, we also require
the potential induced by a scalar density on $\Gamma$. In particular, in
equation~\eqref{nobreakdown}, the term
$\bn \cdot \nabla \times \cS\bbm$ is a scalar function to which a
layer potential operator must be applied. Using equation \eqref{scalar1}-\eqref{scalar2}, the calculation of scalar single layer potential is straightforward. The gradient is then given using the standard form of the gradient in
cylindrical coordinates.

\begin{remark}
  While separation of variables has permitted us to reduce two
  dimensional surface integrals to one dimensional line integrals, the
  kernels $\gm$, $\gccm$ and $\gssm$ defined in~\eqref{eq_green123}
  are not available in closed form. See~\cite{HK2014} for a
  description on how to efficiently evaluate them. In the numerical
  examples of this work, we merely  use adaptive Gaussian quadrature
  for their evaluation. Significant speedups in the resulting code
  could be obtained by optimizing their calculation. Our goal is
  merely to demonstrate the behavior of our novel integral equation
  for scattering from cavities.
\end{remark}

\subsection{Discretization}
\label{sec-disc}

We discretize each smooth segment of the piecewise-smooth boundary by 
a set of panels of uniform length, so that there are at least
12 points per 
wavelength. We then refine the end panels on each segment dyadically until the 
smallest segment is of length $\epsilon$, where $\epsilon$ is the desired 
precision. Each panel is discretized using 10 Gauss-Legendre nodes,
and we utilize $10^{\text{th}}$-order accurate generalized Gaussian
quadratures~\cite{BGR2010} as the basis
for a Nystr\"{o}m  method (which takes into account the logarithmic 
singularity in the kernel). For a review and comparison of various
Nystr\"om-type discretizations, see~\cite{hao_2014}.
 Adaptive Gaussian quadrature is used to compute 
the modal Green's function element by element, as well as for
nearly-singular interactions.

In a slight abuse of notation, 
for an integral operator $\mathcal K$, we will denote by $\mtx{K}_m$
the matrix obtained from discretizing the $m^{\text{th}}$ mode of
$\mathcal K$.
Using the formulas of the previous section, 
we can discretize equation~\eqref{nobreakdown} as:
\begin{equation}
\left( \frac{1}{2}\mtx{I} + ik\mtx{T}\mtx{S}_m^{(1)}\mtx{K}_m
+\mtx{T}\mtx{S}_m^{(2)} \mtx{U} \mtx{S}_m^{(3)} + \mtx{N}_m^{(2)} \mtx{N}_m^{(1)}
- 2\mtx{N}_m^{(4)} \mtx{N}_m^{(3)} + \mtx{N}_m^{(5)} \right) \vct{M}_m
= -\mtx{T}\vct{E}^{\text{inc}}_m
\end{equation}
where the matrices above are discretizations of \emph{a single mode} of 
the continuous operators as follows:
\begin{equation}
\begin{aligned}
\mtx{I} &= \mathcal I, &\quad \mtx{T} &= \bn \times, &\quad 
\mtx{U} &= \bn \cdot, \\
\mtx{S}_m^{(1)} &\approx \cS^H_{C_1,m}, &
\mtx{S}_m^{(2)} &\approx \nabla \cS^H_{C_1,m}, &
\mtx{S}_m^{(3)} &\approx \nabla \times \cS_{\Gamma,m},\\
\mtx{K}_m &\approx \mk_{C_1,\Gamma,m}, &
\mtx{N}_m^{(1)} &\approx \mn_{C_1,\Gamma,m}, & 
\mtx{N}_m^{(2)} &\approx \mn^H_{\Gamma,C_1,m}, \\
\mtx{N}_m^{(3)} &\approx \mn_{B,\Gamma,m}, & 
\mtx{N}_m^{(4)} &\approx \mn_{\Gamma,B,m}, &
\mtx{N}_m^{(5)} &\approx \mn_{\Gamma,\Gamma,m}, \\
\end{aligned}
\end{equation}
and the discretization of the $m^{\text{th}}$ mode of the
solution $\bbm$ is given by $\vct{M}_m$ and the incoming data $\beinc$
is given by $\vct{E}^{\text{inc}}_m$. 
The matrix $\mtx{N}_m^{(5)}$ correspond to a layer potential with 
singular kernel, and is therefore obtained via discretization with generalized
Gaussian quadrature. All other matrices correspond to layer potentials
without singular kernels (no self-interactions) and therefore can be
discretized using smooth and adaptive Gaussian quadrature (near 
geometric refinement). The matrix $\mtx{K}_m$ can be constructed by
discretizing either the operator $\nabla \times \nabla \times \cS$, or
by invoking Lemma~\ref{lemma1}.

\section{Numerical examples}
\label{sec_numeri}

In this section, we illustrate the performance of our algorithm for three 
distinct piecewise smooth cavity structures. Because of the singularities
induced in the densities at the non-smooth junction points, dyadic refinement 
is applied on each segment, as discussed in section~\ref{sec-disc}. 

We normalize the physical length scale so that the cavity can be covered 
by a hemisphere $C_1$ with radius 2, centered at $(0,0,0)$. 
To test the accuracy of the solver, we first create an artificial problem, 
in which the exact solution is known. For this, we choose the field in 
$\mathbb{R}^{3+}\backslash\Omega_1$ to be generated by a current loop 
located in $\Omega_1$, and the field in $\Omega_1$ to be generated by a 
current loop located in $\mathbb{R}^{3+}\backslash\Omega_1$. 
In other words, the exact exterior
field in $\mathbb{R}^{3+}\backslash\Omega_1$ is
\begin{equation}
\begin{aligned}
\beexte &= -\frac{1}{ik}\nabla\times \nabla\times S^H_{\ell_1} \bj_\theta, \\
\bhexte &= \nabla\times S^H_{\ell_1}\bj_\theta,
\end{aligned}
\end{equation}
and the exact interior field in $\Omega_1$ is given by
\begin{equation}
\begin{aligned}
\beinte &= -\frac{1}{ik}\nabla\times \nabla\times S^H_{\ell_2} \bj_\theta, \\
\bhinte &= \nabla\times S^H_{\ell_2}\bj_\theta,
\end{aligned}
\end{equation}
where~$\ell_1$ and~$\ell_2$ are horizontal circular loops with radius
0.5. The loop~$\ell_1$ is located in~$\Omega_1$ and~$\ell_2$
in~$\mathbb{R}^{3+}\backslash\Omega_1$, each with current
density~$\bj =e^{i\theta}\et $. Given the field representation
in~\eqref{eq_rep_ext} and~\eqref{eq_rep_int}, one can introduce the
jump conditions along $C_1$ and the boundary condition
on~$B\cup \Gamma$ consistent with the specified analytic solution.
Note that, although we only use a single azimuthal mode in the current
that defines the exact solution, the number of Fourier modes needed to
resolve the actual field depends on the location of the loops.  To
test only the solver for the~$m=1$ mode, we center~$\ell_1$
at~$(0,0,0.3)$ and~$\ell_2$ at~$(0,0,3.3)$.  To test the full
three-dimensional problem, we place the center of~$\ell_1$
at~$(0,0,0.3)$ but move the second loop off-axis, centering~$\ell_2$
at~$(1.2,0,3.3)$.  We use as many modes as required in order to
resolve the data to precision~$\epsilon$ (which depends in part on the
governing frequency~$k$).

We also solve a true scattering problem with incident plane wave:
\begin{align*}
\Einc &= (\hat{\mathbf d}\times \hat{ \mathbf p}) \times 
\hat{\mathbf d} \, e^{ik\hat{\mathbf d}\cdot \bx}  -
(\hat{ \mathbf d}'\times \hat{ \mathbf p}')\times \hat{ \mathbf d}' \, 
e^{ik\hat{\mathbf d}'\cdot \bx}\\
\Hinc &= \hat{\mathbf d} \times \hat{\mathbf p} \, 
e^{ik\hat{\mathbf d} \cdot \bx} -\hat{\mathbf d}' \times 
\hat{\mathbf p}' \, e^{ik\hat{\mathbf d}'\cdot \bx}
\end{align*}
where $\hat{\mathbf d}$ is the propagation direction,
$\hat{\mathbf p}$ is the polarization, $\hat{\mathbf d}'$ 
and $\hat{\mathbf p}'$ are
the reflected directions with respect to $x_3 = 0$. Through out all
the examples, we choose
\begin{align*}
  \hat{\mathbf d} &= \left(\cos(\pi/4)\sin(\pi/8),\, 
\sin(\pi/4)\sin(\pi/8), \, \cos(\pi/8)\right),  \\
  \hat{\mathbf p} &= \left(\cos(\pi/5)\sin(\pi/10), \,
 \sin(\pi/5)\sin(\pi/10), \, \cos(\pi/10)    \right).
\end{align*}
We make use of the following notation in subsequent tables:
\begin{itemize}
\item $k$: the governing wavenumber,
\item $N_f$: the number of Fourier modes used to resolve the solution,
\item $N_{pts}$: the total number of points used 
to discretize $\Gamma$, $C_1$ and $B$,
\item  $N_{tot}$: the total number of unknowns in the discretized 
integral equation,
\item $T_{matgen}$: the time (secs.) to construct the relevant 
matrix entries for all integral equations,
\item $T_{solve}$: the time  (secs.) to solve the linear system
\item $E_{error}$: the relative $L^2$ error measured at a few random points 
inside the cavity.
\end{itemize}  

 All experiments were implemented in \textsc{Fortran 90} and
carried out on an Intel Xeon 2.5GHz workstation with 
60 cores and 1.5 terabytes of memory. We made use of \textsc{OpenMP} for
parallelism across decoupled Fourier modes, 
and simple block LU-factorization using \textsc{LAPACK} for matrix inversion.
Various fast direct solvers such 
as~\cite{GYM2012,HG2012,Liu2015,JL2014} could be applied if 
larger problems were involved, no effort was made to further
accelerate our code.

\subsection*{Example 1}

\begin{table}[!t]
\center
\caption{Results for rectangular 3D cavity at different wavenumber.}
\label{table_1}
\begin{tabular}{|c|c|c|c|c|c|c|}
\hline
 $k$ & $N_f$ & $N_{pts}$ & $N_{tot}$ & $T_{matgen}(s)$ & $T_{solve}(s)$ & $E_{error}$  \\
\hline
$1$ & 1 &240 & $600$  & $53.5$ & 0.12 & $2 \cdot 10^{-14}$\\
$10$ & 1 &240 & $600$ & $80.9$ & 0.12 & $5 \cdot 10^{-9}$ \\
$10$ & 41 &240 & $600$ & $176$  & 0.25 & $4 \cdot 10^{-9}$ \\
$20$ & 1 &480 & $1200$ & $329.4$ & 1.4 & $6 \cdot 10^{-7}$ \\
$20$ & 41 & 480 & $1200$ & $602.5$ & 1.5 & $4 \cdot 10^{-7}$ \\
$40$ & 1 & 960 & $2400$ & $1918.9$ & 18.9 & $5 \cdot 10^{-6}$ \\
$40$ & 41 &960 & $2400$ & $3545.9$ & 18.4 & $2 \cdot 10^{-5}$ \\
\hline
\end{tabular}
\end{table}

\begin{figure}[!b]
\begin{subfigure}[t]{.35\linewidth}
    \centering
    \includegraphics[scale=0.35]{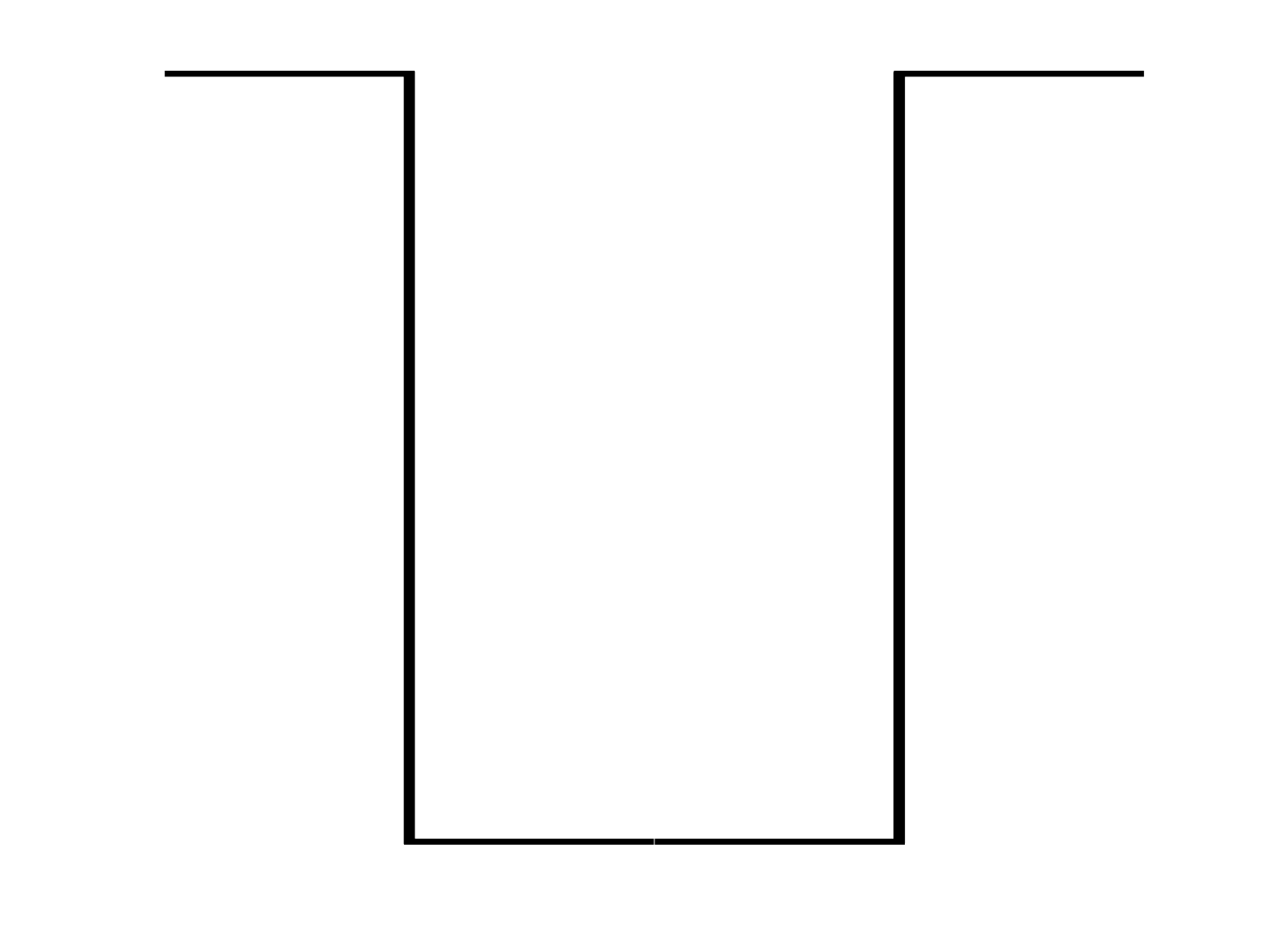} 
    \caption{}\label{fig_3a}
\end{subfigure}
\begin{subfigure}[t]{.35\linewidth}
    \centering
    \includegraphics[scale=0.25]{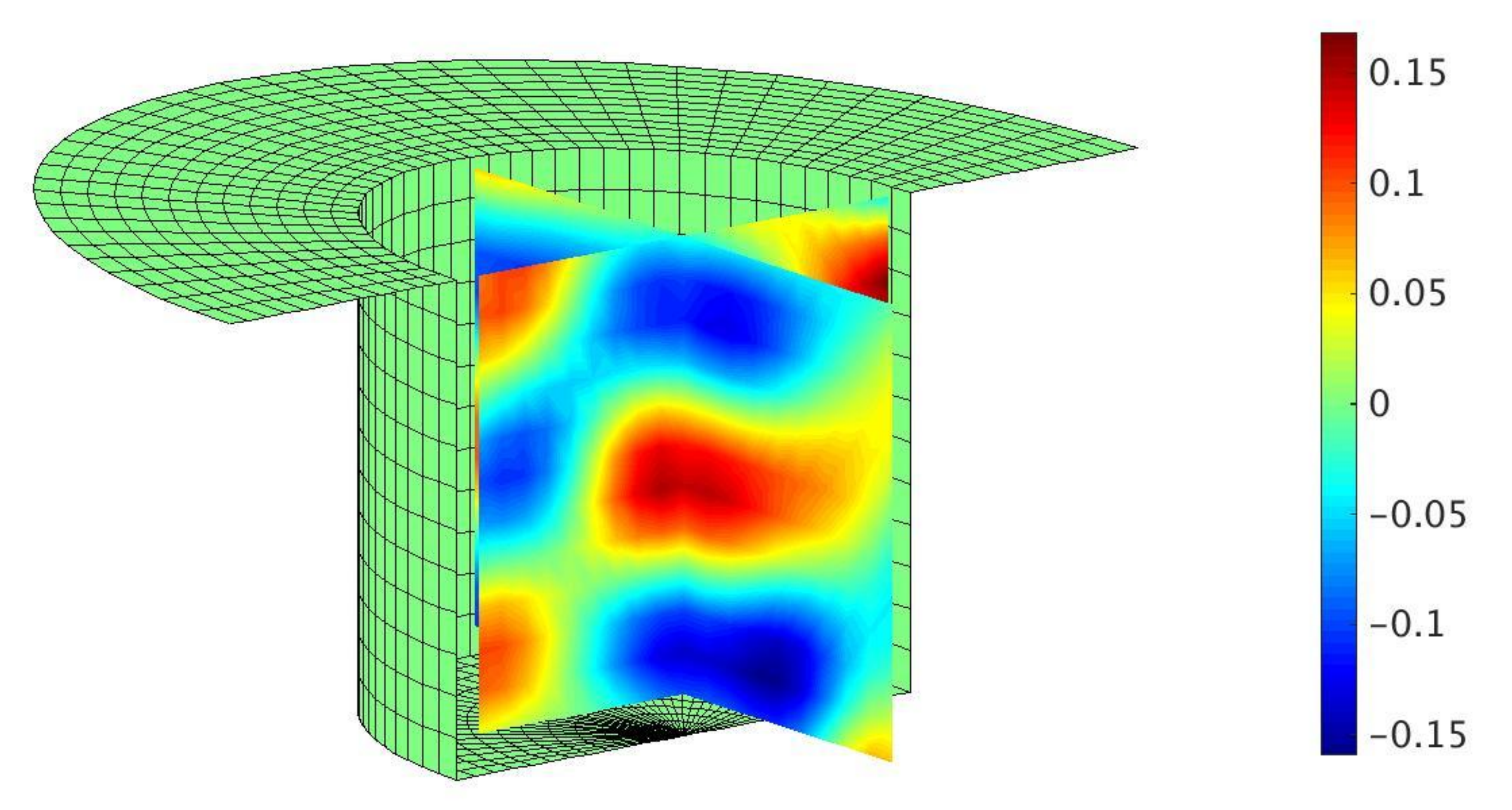}
    \caption{}\label{fig_3b}
  \end{subfigure}    
\caption{Result for example 1. (a) Cross section of the cavity. (b) Real part of the scattered electrical field $E_x$ at $k=10$ from a plane wave incidence.}\label{figure3}
\end{figure}

We first 
consider a cavity with rectangular cross section (Fig.~\ref{fig_3a}). 
The depth and radius of the cavity are both set to $1$. 
We solve the test problem described above at several wavenumbers, with
accuracy results shown in Table~\ref{table_1}. Note that the CPU time is 
dominated by the computation of matrix elements, which scales
quadratically with the number of unknowns.
Because distinct Fourier 
modes are uncoupled, the solution of the various linear systems
is embarassingly parallel and, for the problem sizes considered, does not 
dominate the cost despite the asymptotic $O(N_{pts}^3)$ scaling. 
Note also that the accuracy is very high at low wavenumbers, and slowly
deteriorates at higher wavenumber. The is due largely to the corner 
singularities in the density $\bbm$, 
which are stronger with increasing wavenumber. Additional 
additional refinement is necessary if higher  
accuracy were required. The scattering field for plane wave incidence at $k=10$ is given in Fig.~\ref{fig_3b} with 41 modes.

Results in Table~\ref{table_1} are obtained through solving 
eq.~\eqref{equsys}, without the low-frequency stabilization
of eq.~\eqref{nobreakdown} for $k\ge 1$. In Fig.~\ref{fig_lowfreq}, 
we show the difference in using~\eqref{equsys} or~\eqref{nobreakdown} as
$k \rightarrow 0$. Low-frequency breakdown is clearly manifested
in the original formulation, while~\eqref{nobreakdown} remains stable.

\begin{figure}[!t]
\centering
\includegraphics[width=.9\linewidth]{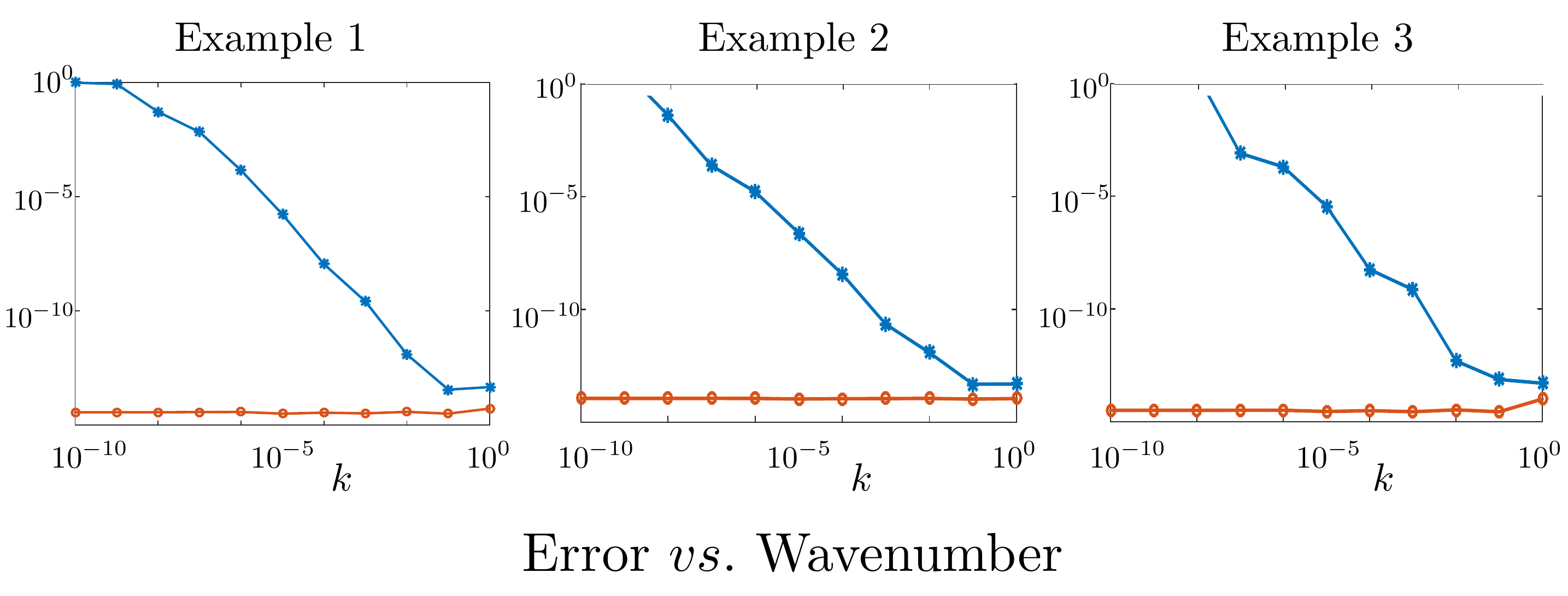}
\caption{Illustration of low-frequency breakdown in the electric
  field.  For Examples 1-3, the lower curve (red) shows the error of
  the solution obtained from~\eqref{nobreakdown} as a function of~$k$
  and the upper curve (blue) shows the error of the solution obtained
  from solving~\eqref{equsys}.}\label{fig_lowfreq}
\end{figure}

\subsection*{Example 2}

For our second example, we consider the cavity
whose generating curve  (Fig.~\ref{fig_4a}) is given by
\begin{equation}
\begin{aligned}
r(s) &= [1-0.1\sin(6\pi s)]\sin(\frac{\pi}{2}s),  \\
z(s) &= -[1-0.1\sin(6\pi s)]\cos(\frac{\pi}{2}s),
\end{aligned} 
\end{equation}
for $s\in [0,1]$.
The incoming field is again generated by the loop source as stated at the beginning of this section.
Accuracy results are provided in Table~\ref{table_2} 
for various wavenumbers. 
A sufficient number of points is used to obtain 
high accuracy at all wavenumbers, with the
computational cost again dominated by  $T_{matgen}$. Fig.~\ref{fig_4b} shows the plane wave scattering at $k=10$ with 41 modes. 
We also compare the behavior of eqs.~\eqref{equsys} and~\eqref{nobreakdown} 
in terms of low-frequency breakdown (Fig.~\ref{fig_lowfreq}) and see the
advantages of~\eqref{nobreakdown} as $k \rightarrow 0$.

\begin{table}[!b]
\center
\caption{Numerical results for the cavity of Example 2
at various wavenumbers.}
\label{table_2}
\begin{tabular}{|c|c|c|c|c|c|c|}
\hline
 $k$ & $N_f$ & $N_{pts}$ & $N_{tot}$ & $T_{matgen}(s)$ & $T_{solve}(s)$ & $E_{error}$  \\
\hline
$1$ & 1 &360 & $960$  & $112.0$ & 0.5 & $4 \cdot 10^{-15}$\\
$10$ & 1 &360 & $960$ & $138.4$ & 0.5 & $5 \cdot 4^{-13}$ \\
$10$ & 41 &360 & $960$ & $273.4$  & 0.5 & $5 \cdot 10^{-13}$ \\
$20$ & 1 &720 & $1920$ & $578.8$ & 9.3 & $4 \cdot 10^{-12}$ \\
$20$ & 61 & 720 & $1920$ & $1237.5$ & 12.3 & $4 \cdot 10^{-12}$ \\
$40$ & 1 & 1440 & $3840$ & $4068.7$ & 134.5 & $5 \cdot 10^{-11}$ \\
$40$ & 81 &1440 & $3840$ & $9137.6$ & 145.1 & $7 \cdot 10^{-11}$ \\
\hline
\end{tabular}
\end{table}

\begin{figure}[!t]
\begin{subfigure}[t]{.35\linewidth}
    \centering
    \includegraphics[scale=0.35]{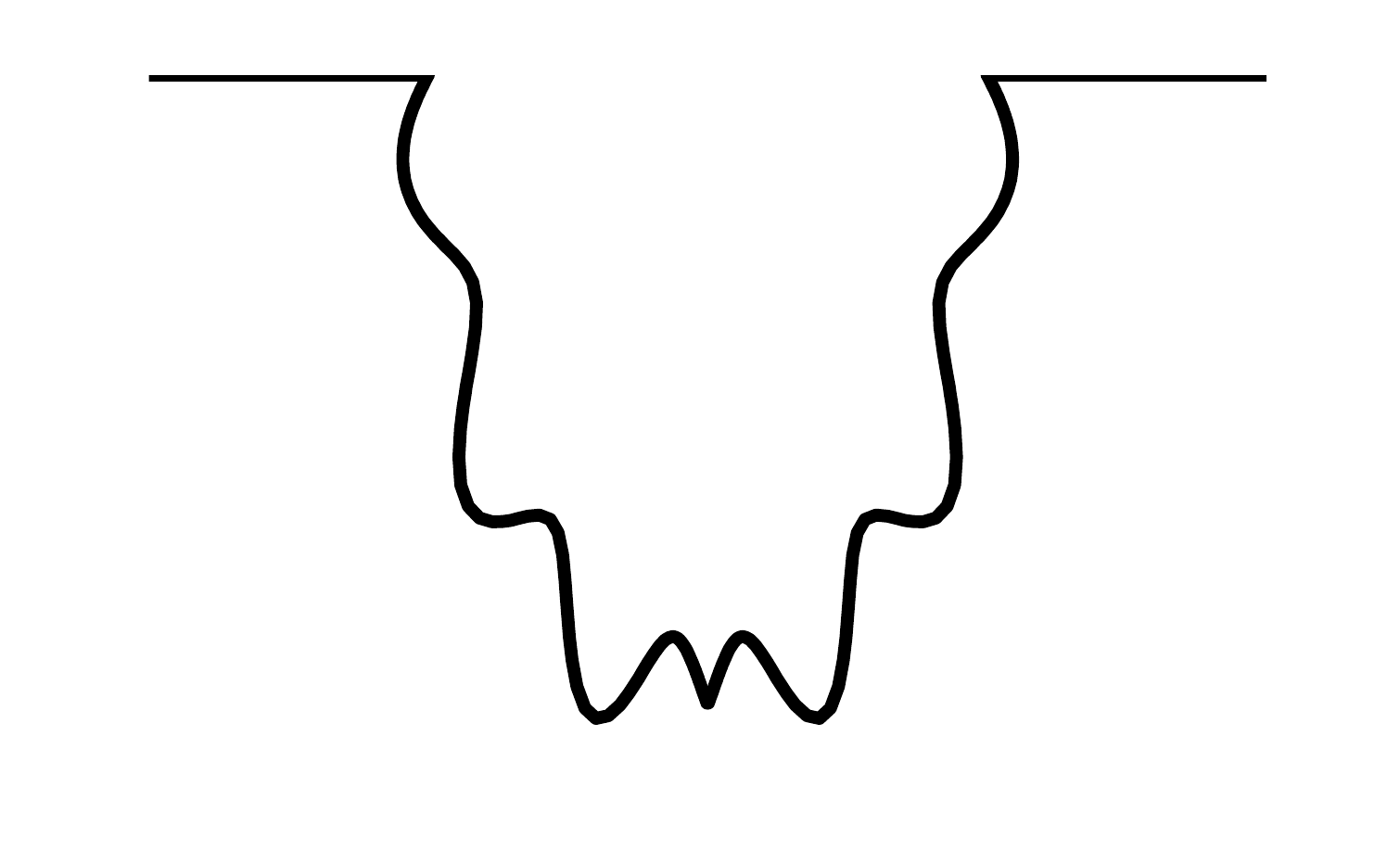} 
    \caption{}\label{fig_4a}
\end{subfigure}
\begin{subfigure}[t]{.35\linewidth}
    \centering
    \includegraphics[scale=0.25]{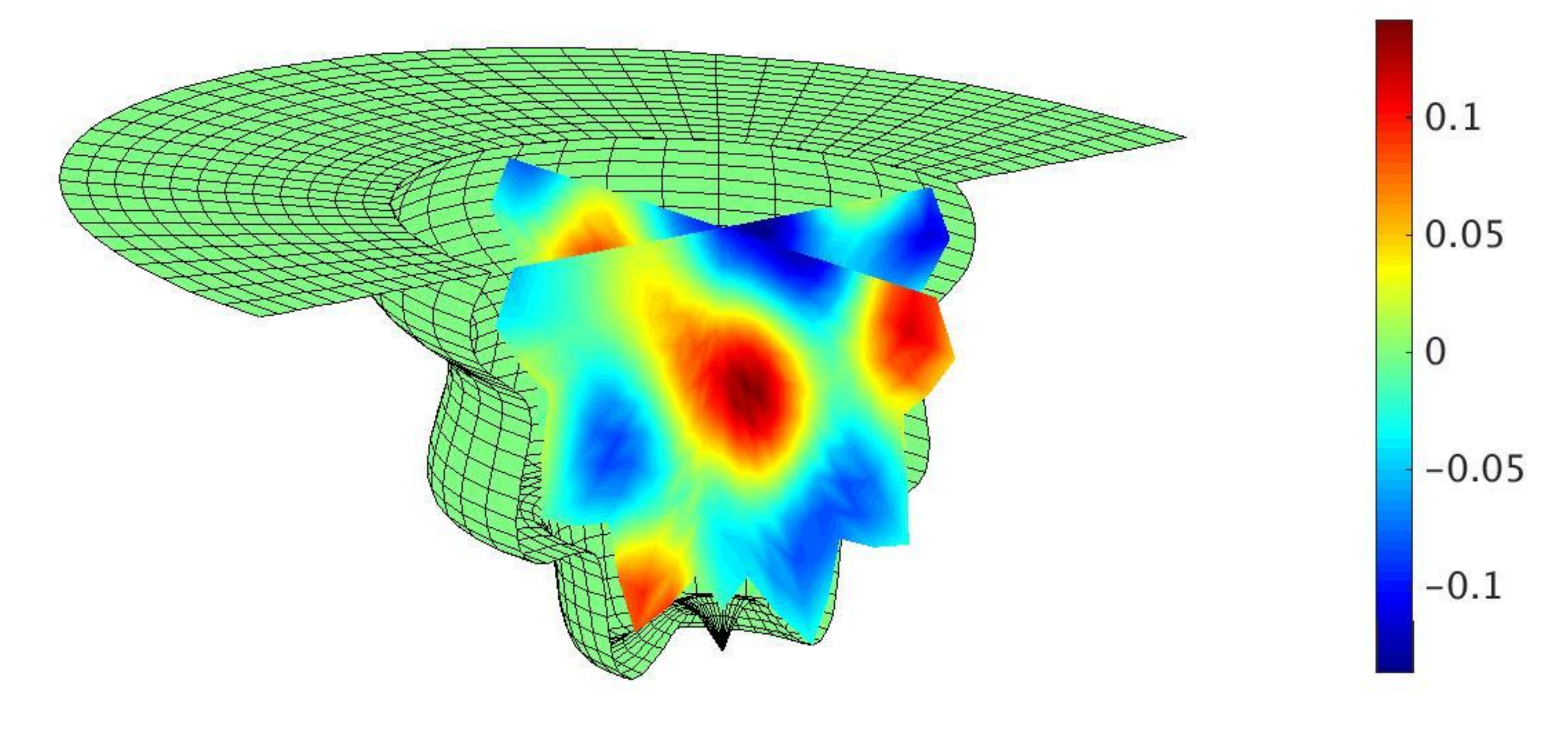}
    \caption{}\label{fig_4b}
  \end{subfigure}    
  \caption{Results for Example 2. (a) Cross section of the cavity. (b)
    Real part of the scattered electric field $E_x$ at
    $k=10$ from a plane wave incidence.}\label{figure4}
\end{figure}

\subsection*{Example 3}

For our final example, we consider the cavity generated by a polygonal 
curve whose vertex coordinates are 
given by
\begin{equation}
V = \left\{
  (0,-0.5), \, (0.5,-0.5), \, (0.5,-0.7), \, (0.25,-0.7), \, 
(0.25,-1), \, (1,-1), \, (1,0) 
\right\},
\end{equation}  
see Figure~\ref{fig_5a}.

We employ dyadic refinement on each segment to resolve the various
corner singularities.
Results are shown in Table~\ref{table_3}, with accuracies given
by comparison with the exact data. For plane wave incidence at $k=10$, Fig.~\ref{fig_5b} gives the scattered field.
The low-frequency behavior is illustrated in Fig.~\ref{fig_lowfreq}, which shows the advantage of equation \eqref{nobreakdown} again.

\begin{figure}[!b]
\begin{subfigure}[t]{.35\linewidth}
    \centering
    \includegraphics[scale=0.35]{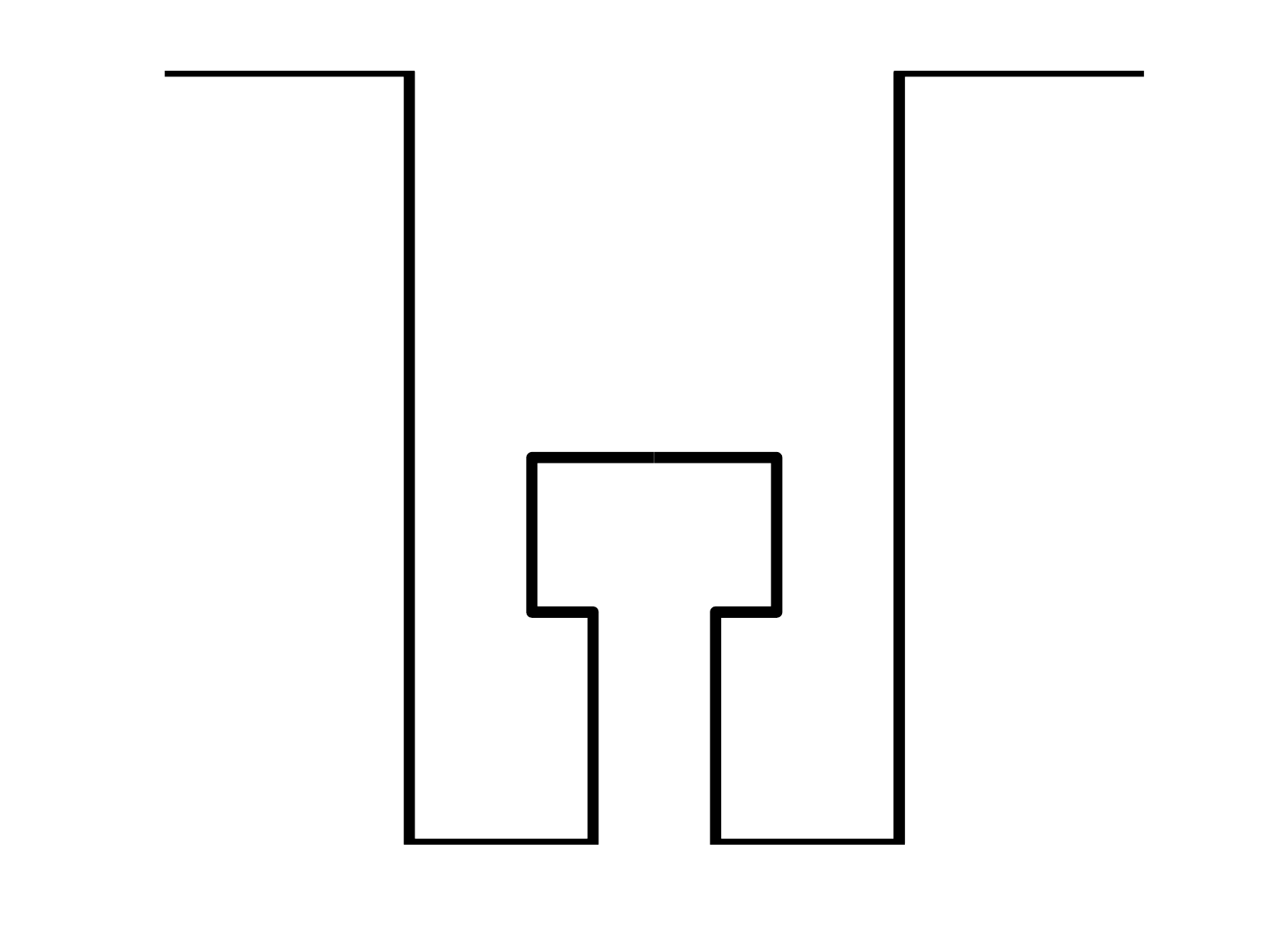} 
    \caption{}\label{fig_5a}
\end{subfigure}
\begin{subfigure}[t]{.35\linewidth}
    \centering
    \includegraphics[scale=0.25]{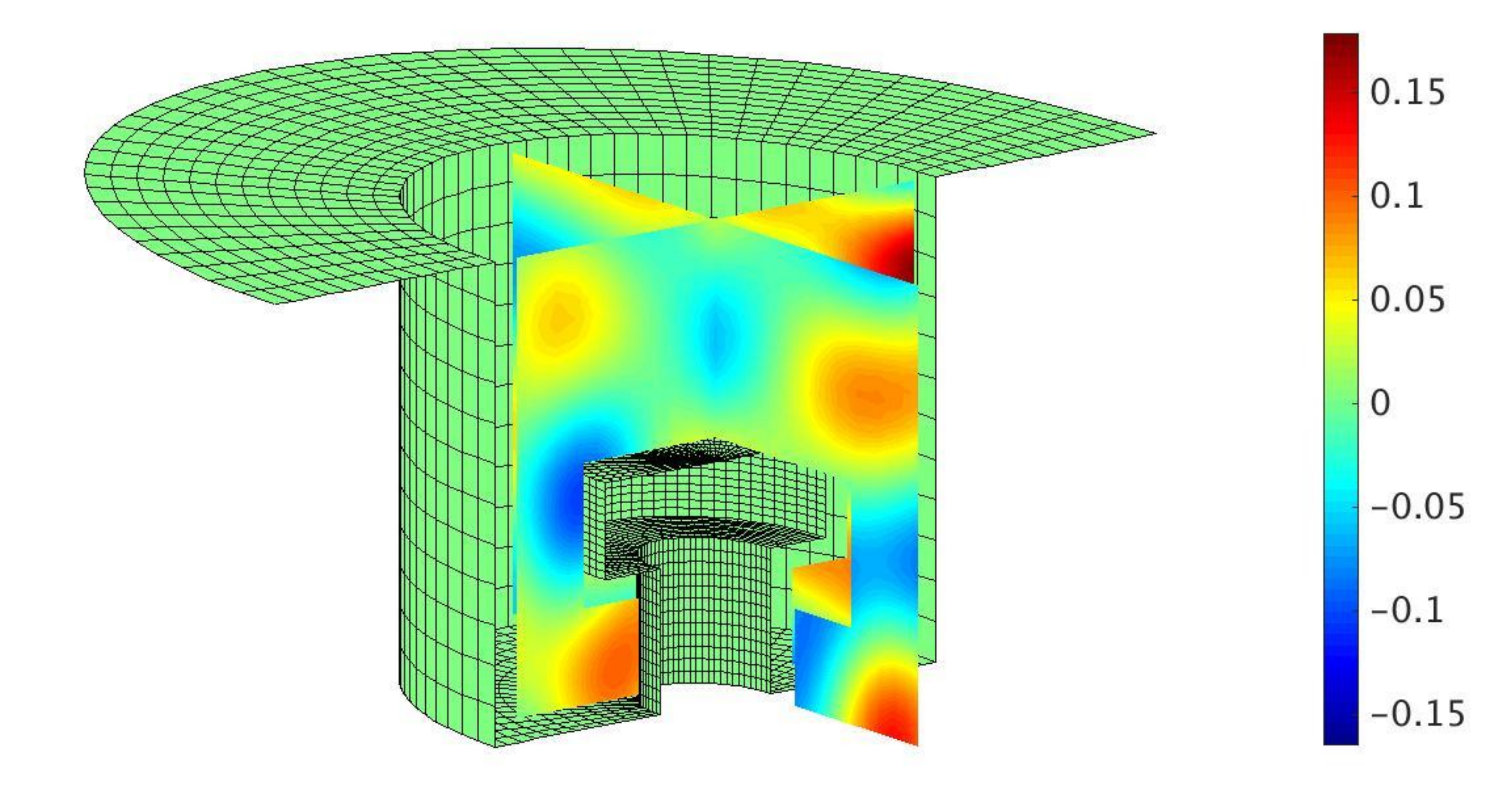}
    \caption{}\label{fig_5b}
  \end{subfigure}    
  \caption{Results for Example 3. (a) Cross section of the cavity. (b)
    Real part of the scattered electrical field $E_x$ for
    $k=10$ from a plane wave incidence.}\label{figure5}
\end{figure}

\begin{table}[!t]
\center
\caption{Numerical results for the cavity of Example 3
at various wavenumbers.}
\label{table_3}
\begin{tabular}{|c|c|c|c|c|c|c|}
\hline
 $k$ & $N_f$ & $N_{pts}$ &$N_{tot}$ & $T_{matgen}(s)$ & $T_{solve}(s)$ & $E_{error}$  \\
\hline
$1$ & 1 & 480 & $1080$  & $201.2$ & 0.70 & $2 \cdot 10^{-14}$\\
$10$ & 1 & 480 & $1080$ & $246.7$ & 0.70 & $5 \cdot 10^{-9}$ \\
$10$ & 41 &480 & $1080$ & $558.5$  & 0.75 & $4 \cdot 10^{-9}$ \\
$20$ & 1 &960 & $2160$ & $1217.8$ & 13.9 & $6 \cdot 10^{-7}$ \\
$20$ & 41 &960 & $2160$ & $2425.8$ & 13.2 & $4 \cdot 10^{-7}$ \\
$40$ & 1 &1920 & $4320$ & $5252.7$ & 119.6 & $5 \cdot 10^{-6}$ \\
\hline
\end{tabular}
\end{table}

\section{Conclusions}
\label{sec_con}

In this paper, we have developed a new integral representation for 
the problem of scattering from a three-dimensional cavity embedded in a
perfectly conducting half-space which leads to a well conditioned integral 
equation.
The resulting integral equation is resonance free for all
wavenumbers $k$,
immune from
low-frequency (dense-mesh) breakdown, and we have established
existence and uniqueness for its solution. In particular, the
resulting linear system is well-conditioned all the way down to the
static limit.
Furthermore, the solution to this integral equation allows for 
the accurate reconstruction of
the electric field in the limit as $k \to 0$. However,
since inherently the unknowns in our formulation are current-like
vector fields, reconstruction of the magnetic field suffers (albeit
only mildly) from low-frequency breakdown. In order to overcome this,
alternative representations using charge-like variables would have to
be developed.

The effectiveness of the scheme was demonstrated in rotationally
symmetric cavities using separation of variables
in the azimuthal direction, and a subsequent
high-order integral equation method on the cavity's generating curve.
The solution for each mode involves only line integrals along the
generating curve that defines the geometry. This permits the 
use of efficient generalized Gaussian quadratures and
stable, adaptive mesh refinement into the geometric singularities.
We illustrated the performance of the scheme with several examples.

As discussed in section~\ref{sec_lowfreq}, one open question concerns
a mild form of low-frequency breakdown in evaluating the magnetic field.
We are presently investigating whether the use of generalized Debye 
sources~\cite{EG10,EG13} can be used to overcome this
issue. Furthermore, in the
present paper, we have made strong use of axisymmetry in developing
a numerical solver. We are working on extending the relevant code to
arbitrarily-shaped cavities using fully three-dimensional quadratures
on triangulated surfaces.
\bibliographystyle{abbrv}
\bibliography{}  
\end{document}